\documentclass[12pt]{amsart}
\usepackage {amssymb}
\usepackage {amsmath}
\usepackage{amsthm}
\usepackage{graphicx}
\usepackage {amscd}
\usepackage {epic}
\usepackage[alphabetic]{amsrefs}
\usepackage[colorlinks, linkcolor=blue, citecolor=blue, urlcolor=blue]{hyperref}
\usepackage[left=1.5in,right=1.5in,top=1.5in,bottom=1.5in]{geometry}
\usepackage{soul}
\usepackage{mathrsfs}
\usepackage{comment}
\usepackage[all]{xy}
\usepackage{amsmath}
\usepackage{xcolor}
\usepackage{amsthm}
\usepackage{mathrsfs}
\usepackage{tikz-cd}
\usepackage{enumerate}

\newcommand{\set}[1]{\{#1\}}

\newcommand{\act}{\curvearrowright}

\newtheorem{thm}{Theorem}[section]
\newtheorem{prop}[thm]{Proposition}
\newtheorem{lem}[thm]{Lemma}
\newtheorem{cor}[thm]{Corollary}
\newtheorem{exm}[thm]{Example}

\newtheorem{conj}{Conjecture}[section]

\newtheorem{remark}[thm]{Remark}

\theoremstyle{definition}
\newtheorem{df}[thm]{Definition}

\newcommand{\p}{\partial}

\newcommand{\dt}{\Delta}
\newcommand{\bd}{\textbf }

\title[Abundance for K\"ahler Varieties via Algebraic reduction]{An approach to the abundance conjecture for K\"ahler varieties via algebraic reduction}

\author{Zhiyuan Jiang}

\begin{document}

\begin{abstract}
In this article, we establish a strategy to the abundance conjecture for K\"ahler varieties via induction on algebraic dimension. Our strategy is to reduce the abundance conjecture for K\"ahler varieties to the abundance conjecture for projective varieties using the algebraic reduction fibration. In dimension 4, we apply our inductive strategy to obtain some cases of the abundance conjecture for K\"ahler fourfolds that are not algebraic or have trivial $K_X$.
\end{abstract}

\maketitle

\tableofcontents

\section{Introduction}

The abundance conjecture is one of the most important open problems in birational geometry.

\begin{conj}
    Let $(X,\dt)$ be a K\"ahler lc pair. If $K_X+\dt$ is nef, then $K_X+\dt$ is semiample.
\end{conj}

The abundance conjecture was originally stated for projective varieties in \cite[6-1-14]{MR946243}. It predicts that every minimal model admits a globally defined Iitaka fibration, which is crucial in the birational classification theory of algebraic varieties. The conjecture implies other major open conjectures in birational geometry, for example, Iitaka's conjecture on the subadditivity of Kodaira dimensions (see \cite{MR814013}). 

The abundance conjecture is open in dimension $\ge 4$. The  conjecture is known for complex surfaces in \cite{MR731524} and for complex projective varieties in dimensions at most 3 in \cite{miyaoka1}, \cite{miyaoka2} and \cite{MR1161091}. The log abundance theorem for projective threefolds is proven in \cite{KMM94}. There are still major open cases even in dimension 4

In higher dimensions, a major strategy towards the abundance conjecture is to establish an inductive proof using various fibrations. In \cite[Theorem 3.3]{9}, Takao Fujita first applies this method to elliptic fibrations to establish abundance for elliptic threefolds.
Florin Ambro establishes an inductive approach using the nef reduction map in \cite{5}. In \cite[Theorem 4.5]{MR2740690}, Osamu Fujino develops an inductive argument using the Albanese fibration. Based on this inductive argument, Osamu Fujino establishes abundance for projective fourfolds  with irregularity $q\ne 0$ in \cite[Corollary 4.7]{MR2740690}.

One recent breakthrough in bimeromorphic geometry is the establishment of the minimal model program for K\"ahler threefolds \cite{1,12,MR1471137,MR1637984,MR1881199,MR3329195}. In \cite{3},  the minimal model program for K\"ahler fourfolds is established assuming the non-vanishing conjecture for Kähler fourfolds. The MMP for projective morphisms between analytic spaces is systematically discussed in \cite{fujino2022minimalmodelprogramprojective}. In dimension 3, abundance and log abundance for Kähler threefolds are established by Campana, Höring, Peternell, Das and Ou in \cite{13}, \cite{MR4954870} and \cite{14,das2025logabundancecompactkahler}. 

As the next step, the abundance conjecture is conjectured to hold more generally for Kähler varieties of arbitrary dimension.
The abundance conjecture for Kähler varieties is a major open problem in Kähler geometry. Together with the minimal model program, by \cite{MR2257317}, the abundance conjecture would imply deformation invariance of plurigenera for Kähler varieties (which is still open).

In this article, we establish a new inductive approach to the abundance conjecture for Kähler varieties. Our idea is to establish an inductive argument via the algebraic reduction map. For the statement below, recall that the algebraic reduction map of a compact analytic variety $X$ is a meromorphic map from $X$ to a projective variety $Y$ that induces an isomorphism on fields of meromorphic functions, and the algebraic dimension of $X$ is $a(X) := \operatorname{trdeg}_\mathbb C(\mathbb C(X)) = \dim(Y)$ (See Section 2.3 for details).

The main result of this paper is the following inductive approach towards the abundance conjecture for Kähler varieties.

\begin{thm}\label{thmA}
    Let $(X,\dt)$ be a K\"ahler klt log pair. Assume
    \begin{enumerate}
        \item The MMP and abundance for projective varieties hold in dimension $a(X)$ and abundance for K\"ahler varieties holds in dimension $\operatorname{dim}(X) - a(X)$;
        \item The algebraic reduction map of $X$ is almost holomorphic; and
        \item The moduli b-divisor $\bd M$ of the algebraic reduction fibration is b-good and nef.
    \end{enumerate}
    \noindent If $K_X+\dt$ is nef, then $K_X+\dt$ is semiample.
\end{thm}

In dimension 4, we apply our inductive strategy to obtain the following cases of the abundance conjecture for Kähler fourfolds. As far as we are aware, these are the first cases of the abundance conjecture for K\"ahler fourfolds that are not algebraic or have trivial $K_X$.

\begin{thm}\label{dim4}
    Let $X$ be a 4-dimensional compact K\"ahler manifold with $a(X) = 3$. If $K_X$ is nef, then $K_X$ is semiample. 
\end{thm}

By applying Theorem \ref{thmA},  we also recover the special case of the results in \cite{13}, \cite{MR4954870} and \cite{14,das2025logabundancecompactkahler}  when $a(X)\ne 0$:

\begin{thm}\label{dim3}
    Let $(X,\Delta)$ be a klt threefold with $a(X) \ne 0$. If $K_X+\dt$ is nef, then $K_X+\dt$ is semiample.
\end{thm}

We briefly discuss the assumptions in Theorem 1.1 and some cases where these assumptions are known to hold.

The assumption (1) has an inductive nature, i.e. we split the abundance conjecture for K\"ahler varieties in dimension $n$ to the ones in dimension $a(X)$ for projective varieties and $\dim (X)-a(X)$ for K\"ahler varieties.

In assumption (3),  we say a b-divisor $\bd D$ is \textbf{b-nef and good}  if there exists a birational model $Y'\to Y$ and a proper contraction $h:Y'\to Z$ such that $\bd D$ descends to $Y'$ and $\bd D_{Y'}\sim _\mathbb Q h^*N$ for some nef and big $\mathbb Q$-divisor $H$ of $Z$.

Assumption (3) has been fully established for projective varieties in \cite[Theorem 3.3]{15}. In \cite[Proposition 4.16]{loginov2025finitenessprojectivepluricanonicalrepresentation}, Loginov and Shramov prove that every klt-trivial fibration of relative dimension one must have a b-semiample moduli b-divisor. In particular, this implies that assumption (3) holds in the case $a(X) =\dim(X)-1$.

For assumption (2), since the question on almost holomorphicity is interesting on its own, we would like to state it as a conjecture:

\begin{conj}
    Let $X$ be a compact K\"ahler variety. Then the algebraic reduction map of $X$ is almost holomorphic.
\end{conj}

The conjecture is known to be true when $\dim(X) = 3$ \cite[Corollary 1.3]{MR2103314} or when $X$ is smooth and $a(X) = \dim X-1$ \cite[Lemma 2.12]{MR4698899}. In fact, more generally, the conjecture holds when the algebraic reduction map has projective fibers, or when the target is not uniruled. The author would like to thank Professor Frédéric Campana for discussing this question with us and providing the last statement.

To prove theorem \ref{thmA}, our strategy is to reduce abundance for K\"ahler varieties to abundance for projective varieties via the algebraic reduction map, which is a map that exists for every analytic variety and is unique up to bimeormorphic equivalence.

The idea is to compare the canonical divisors along the algebraic reduction map. To make this possible, we establish the following results
\begin{itemize}
    \item The pullback compatibility of Zariski decomposition for Kähler varieties (Proposition \ref{pullback}): It provides us with a bimeromorphic description of the abundance conjecture.
    \item The canonical bundle formula for morphisms from Kähler varieties to projective varieties (Theorem \ref{moduli_b-nef}): This is used to compare the canonical bundles along the algebraic reduction fibration.
    \item Semi-stable reduction in codimension one for morphisms from Kähler varieties to projective varieties (Theorem \ref{SSR}): This is used to modify the klt-trivial fibration in the proof of the canonical bundle formula.
\end{itemize}

The pullback compatibility of Zariski decomposition is established in \cite[Proposition 1.24]{9} for projective varieties. In his proof of the main Lemma, Takao Fujita makes use of hyperplane sections which do not always exist in the analytic set-up. In this article, we develop a new Hodge-theoretic strategy to prove an analytic analogue of Fujita's compatibility result.

The canonical bundle formula is proven in the fundamental paper \cite[Theorem 2.7]{8} for projective varieties.  Our observation is that morphisms from K\"ahler varieties to projective varieties provide a suitable setup for the proof to work.  We study the roles that the domain and base play in the theory and extend the canonical bundle formula to the case where the domain is K\"ahler and base is projective.

The semistable reduction in codimension one is established in \cite{Kempf1973} and \cite[Theorem 4.3]{8} for projective varieties. We extend the semistable reduction in codimension one to the set-up where the domain is analytic and the base is projective. In the proof, we study the issue of extension of coherent ideal sheaves from open subsets to make sure the necessary extension arguments apply in the analytic set-up.

Finally, using those tools, we prove Theorem \ref{thmA} in Section 6 by comparing the canonical bundles along the algebraic reduction fibration.

\section*{Acknowledgment}

I would like to thank Professor James M\textsuperscript cKernan for his guidance, discussion and support. I would like to thank Professor Takumi Murayama for discussions and his support. I'm also grateful to Professor Florin Ambro for very insightful discussions. I also want to thank Professor Frédéric Campana, Professor Paolo Cascini, Professor Andreas H\"oring for many very helpful discussions. The author was partially supported by a grant from the Simons Foundation.

\section{Preliminaries}

\subsection{Analytic Geometry}

In this paper, an analytic variety is a (Hausdorff and second countable) complex analytic space that is both reduced and irreducible. We denote the nonsingular locus of an analytic variety $X$ as $X^{sm}$.  Let $X$ be an analytic variety and $U\subseteq X$ be Zariski open. We say $U$ is a \textbf{big Zariski open subset}  if $X-U$ is a nowhere dense closed subset of codimension at least $2$. We denote $\mathscr K_X$ as the sheaf of meromorphic functions. We will also apply the GAGA principle without explicitly mentioning it.

\begin{df}
    A \textbf{contraction} is a proper surjective morphism $f:X\rightarrow Y$ such that $f_*\mathscr O_X = \mathscr O_Y$.
\end{df}

\begin{df}
    Let $f:X\dashrightarrow Y$ be a meromorphic map between compact varieties with domain $U$. We say $f$ is \textbf{almost holomorphic} if there is a $y\in Y$ such that $f^{-1}(y)$ is compact and contained in $U$.
\end{df}

We will need the following results:

\begin{thm}[Flattening Theorem]\cite[Theorem 4.4]{hironaka1975flattening}
    Let $f:X\rightarrow Y$ be a surjective morphism of compact varieties. Then there is a commutative diagram:
    $$\xymatrix{
    X'\ar[r]^\nu\ar[d]_{f'}&X\ar[d]^f\\
    Y'\ar[r]^\pi& Y
    }$$
    where $\pi,\nu$ are bimeromorphic morphisms and $f'$ is flat. We call $f'$ a \textbf{flat model}  of $f$.
\end{thm}

\begin{remark}
    Indeed, $\pi$ can be constructed as a succession of blowups along nonsingular centers. Thus we can always assume $\pi,\nu$ to be projective morphisms.
\end{remark}

For a morphism $f:X\to Y$ and an open subset $U\subseteq Y$, we write $X_U = f^{-1}(U)$.

\begin{prop}\cite[Lemma 2.24]{Analytic-SSR}\label{finite-cover-uniqueness}
    Let $X$ be a normal analytic variety and $U\subseteq X$ be a dense Zariski open. Let $\pi_i:X_i\to X$ be finite surjective morphisms from normal analytic varieties and $(X_1)_U\to (X_2)_U$ be an isomorphism over $U$. Then it extends uniquely to an isomorphism over $X$.
\end{prop}

\subsection{Positivity in K\"ahler Geometry}

In this section, we collect some basic notions of bimeromorphic geometry.  More detailed discussions can be found in \cite{1} and \cite{3}.

\begin{df}
    An analytic variety $X$ is said to be a \textbf{K\"ahler variety}  if $X$ admits a positive closed $(1,1)$-form $\omega$,  called the \textbf{K\"ahler form}, such that for every point $x\in X$, there exists an open neighborhood $x\in U\subseteq X$ and a closed immersion $i:U\to V$ where $V\subseteq \mathbb C^N$ is an open subset, and a strictly plurisubharmonic smooth function $f:V\to\mathbb C$ with $\omega_{U\cap X^{sm}} = (i\p\overline{\p}f)|_{U\cap X^{sm}}$. 
\end{df}

\begin{prop}\cite[Remark 2.3]{1}
    Let $X$ be a compact K\"ahler variety.
    \begin{enumerate}
        \item Any subvariety of $X$ is K\"ahler;
        \item If there is  a projective morphism  $Y\to X$, then $Y$ is K\"ahler.
    \end{enumerate}
\end{prop}

\begin{remark}\label{kaehlerness-perserving}
    As a result, we can perform most bimeromorphic operations, such as base changes and resolutions of singularities, without leaving the K\"ahler category. We will implicitly use this property throughout the paper.
\end{remark}

\begin{df}
    A class $u\in H^{1,1}(X,\mathbb R)$ is called a \textbf{nef} class if there is a  $(1,1)$-form $\omega> 0$ such that for every $\epsilon>0$, there is a representative $\alpha$ of the class $u$ satisfying $\alpha+\epsilon\omega\geq 0$. 
\end{df}

As one can expect, this definition is compatible with the usual notion of being nef for projective varieties: Let $D$ be a divisor on a projective manifold $X$. Then $D$ is algebraically nef if and only if $D$ is analytically nef \cite[6.10]{2}.

Many formal properties of nefness still hold for K\"ahler varieties. We will need the pullback property and subvariety test criterion:

\begin{thm}\label{pullbacknef}\cite[2.38]{3}
    Let $\pi:X\rightarrow Y$ be a proper surjective morphism between compact K\"ahler varieties. Then $\alpha$ is nef if and only if $\pi^*\alpha$ is nef.
\end{thm}

\begin{df}
    A class $\alpha\in H^{1,1}(X,\mathbb R)$ is called \textbf{pseudoeffective} if $\alpha$ has a representative $T$ that is a  closed positive $(1,1)$-current, that is $[T] = \alpha$.  Let $\mathscr E(X)$ denote the pseudoeffective cone. 
\end{df}

In general, on an analytic variety, curves are scarce. As a result, instead of testing nefness using curves, one should use subvarieties:

\begin{thm}\label{neftest}\cite[2.36, 2.37]{3}
    Let $X$ be a compact analytic variety and $\alpha$ be a closed $C^\infty$-$(1,1)$-form. Then $\alpha$ is nef if and only if $\alpha|_Z$ is pseudoeffective for all subvarieties $Z\subseteq X$.
\end{thm}

\subsection{Algebraic Reduction}

In this section, we will briefly go over some basic properties of the algebraic reduction map. The readers can find a detailed discussion in \cite[Section 3]{ueno75}.

Let $X$ be a compact analytic variety. The notion of algebraic dimension describes the richness of meromorphic functions on an analytic variety:

\begin{df}
    The function field $\mathbb C(X)$ of $X$ is the field of all meromorphic functions.
    The \textbf{algebraic dimension}  of $X$ is defined as
    \[
        a(X) := \operatorname{trdeg}_\mathbb C(\mathbb C(X)).
    \]
\end{df}

By definition of meromorphic functions, the function field $\mathbb C(X)$ is a bimeromorphic invariant of compact analytic spaces.

\begin{remark}
    In contrast to birational geometry,  function fields are no longer bimeromorphic invariants of analytic varieties that are not compact. For example, $\mathbb C(\mathbb C)\neq \mathbb C(\mathbb CP^1)$. This is because functions may have essential singularities.
\end{remark}

The algebraic dimension describe how far an analytic variety is from being algebraic:

\begin{prop}
    Let $X$ be a compact analytic variety. 

    (1) $\mathbb C(X)/\mathbb C$ is a finitely generated extension. Furthermore, we have $$a(X)\leq \operatorname{dim}(M)$$

    (2) If $X$ is a K\"ahler manifold, then $X$ is projective if and only if $a(X) = \operatorname{dim}_\mathbb C(X)$.
\end{prop}

\begin{proof}
    (1) is \cite[Theorem 3.1]{ueno75} and (2) is \cite[Remark 3.7(2)]{ueno75}.
\end{proof}

Statement (1) shows there is a projective variety with rational field $\mathbb C(X)$. This projective variety is birationally unique because the birational type of an algebraic variety is entirely determined by its function field. In fact, there is a natural map from $X$ to such a projective variety, called the algebraic reduction map:

\begin{prop}\cite[Definition 3.2]{ueno75}
    For a compact analytic variety $X$, there is a diagram:
    $$\xymatrix{
    X&M\ar[l]_\mu\ar[d]^\psi\\
    &V
    }$$
    where $M$ is a compact complex manifold, and $V$ is a nonsingular projective variety of dimension $a(X)$ such that:

    (1) $\psi$ is surjective and $\mu$ is bimeromorphic;

    (2) $\psi^*:\mathbb C(V)\rightarrow \mathbb C(M)$ is an isomorphism.
\end{prop}

\begin{df}
    For a compact analytic variety $X$, we call any meromorphic map $X\dashrightarrow V$ which satisfies the properties in the previous proposition an \textbf{algebraic reduction} map.  We call a diagram as in the previous proposition as an algebraic reduction diagram.
\end{df}

\begin{remark}
    If $\psi:M\rightarrow V$ is the fibration in the algebraic reduction diagram in the proposition, by Zariski's main theorem, $\psi$ is a contraction morphism, i.e. $\psi$ has connected fibers (\cite[Proposition 3.4]{ueno75}).
\end{remark}

\section{Zariski Decomposition}

The abundance conjecture is originally stated for minimal models.  This is inconvenient, since in practice we make bimeromorphic modifications very often. To give the abundance conjecture a bimeromorphic description, we study the theory of Zariski decompositions in this section.

We first show that the existence of Zariski decompositions can be reduced to the category of nonsingular K\"ahler varieties (Corollary \ref{smooth-category}). Then we will prove the main theorem of this section, the pullback compatibility property, in Proposition \ref{pullback}.

The original proof of the pullback compatibility for projective varieties in \cite{9} makes use of hyperplane sections which does not work in the analytic set-up. We use a new Hodge-theoretic approach to prove it in the analytic setting.

\subsection{b-divisors}

In this section, we will discuss the notion of b-divisors, introduced by Shokurov in \cite[Section 1]{Shokurov-b-divisor}. We will restrict our discussion to projective bimeromorphic morphisms rather than proper ones, but the resulting notion of a b-divisor  is equivalent to the one in \cite[Section 1]{Shokurov-b-divisor} (See Remark \ref{proper-vs-projective}).

Fix an ambient category $\mathcal B$, whose objects are  projective bimeromorphic morphisms and morphisms are commutative squares consists of projective bimeromorphic morphisms.
 
Let $X$ be a compact normal analytic variety. We write $\mathcal B_X$ as the fiber category of all isomorphic classes of projective bimeromorphic morphisms $Y\to X$. We will write $\tilde{\mathcal B}_X$ be the full subcategory of $\mathcal B_X$ whose objects $Y\to X$ are morphisms from compact manifolds compact manifolds. We call any object $(\pi_Y:Y\to X)\in \mathcal B_X$  a \textbf{model}. We write $(Y_1\to X)\ge (Y_2\to X)$ if there exists a projective bimeromorphic morphism $Y_1\to Y_2$ over $X$.

\begin{lem}
    Let $X$ be a compact normal analytic variety. Then $\mathcal B_X$ is directed, i.e. every finite many of objects have a common upper bound in $\mathcal B_X$. If moreover $X$ is nonsingular, then $\tilde{\mathcal B}_X$ is directed.
\end{lem}

\begin{proof}
    For two models $Y_i\to X$, we know $Y_1\times_X Y_2\to X$ followed by a projective resolution of the main component is a model of $X$ which dominates both $Y_1$ and $Y_2$.
\end{proof}

\begin{remark}
    If $X$ is a compact K\"ahler variety, by Remark \ref{kaehlerness-perserving}, any model $Y\in\mathcal B_X$ is K\"ahler as well. 
\end{remark}

\begin{df}
    Let $X$ be a compact normal analytic variety. A \textbf{b-divisor}  $\bd D$ on $X$ is an element in
    \[
        \bd D\in \lim_{Y\in \mathcal B_X} \operatorname{Div}(Y)
    \]
    where the transition maps are induced by proper pushforward.  We denote $\bd{Div}(X)$ as the space of all b-divisors on $X$.
\end{df}

A \textbf{filter basis}  of $\mathcal B$ is a subcategory $\mathcal B'$ of $\mathcal B$ such that for any $x\in \mathcal B$, there is a $y\in \mathcal B'$ with $y\ge x$. 
It follows from the formal property that if a b-divisor $\bd D$ is defined on a filter basis $\mathcal B'$ of $\mathcal B_X$, then $\bd D$ extends to a b-divisor on $\mathcal B_X$. 
Hence, for any projective bimeromorphic morphism $Y\to X$ between normal compact varieties, there is a natural identification $\bd{Div}(X) = \bd{Div}(Y)$.

\begin{remark}\label{proper-vs-projective}
     Though our definition is for projective bimeromorphic morphisms, it is equivalent to the definition in \cite[Section 1]{Shokurov-b-divisor}. In fact, by \cite[Corollary 2]{hironaka1975flattening},  projective bimeromorphic models form a filter basis of the category of proper bimeromorphic models.
\end{remark}

The following filter bases are useful for our purpose.

\begin{enumerate}
        \item If $X$ is a compact manifold, then $\tilde{\mathcal B}_X$ is a filter basis of $\mathcal B_X$;
        \item If $Y\to X$ is a projective bimeromorphic morphism between normal compact varieties, then $\mathcal B_{Y}$ is a filter basis of $\mathcal B_{X}$.
\end{enumerate}

If $D$ is a Cartier divisor on $X$, then we define the \textbf{closure} of $D$, denoted by $\overline D$, as the b-divisor given by $\overline{D}_Y = \pi^*_YD$. We say a b-divisor $\bd D$ \textbf{descends to} $Y$ if $\bd D = \overline{\bd D_Y}$. 

For a surjective morphism $f:Y\to X$ and a b-Cartier divisor $\bd D$ on $X$, the pullback b-Cartier divisor $f^*\bd D$ is well-defined. In fact, assume $\bd D$ descends to $\tilde X$. We  choose a commutative diagram 
\[
    \begin{tikzcd}
        \tilde Y\arrow{r}{\nu_Y}\arrow{d}{\tilde f} & Y\arrow{d}{f}\\
        \tilde X\arrow{r}{\nu_X} & X
    \end{tikzcd}
\]
such that $\nu_X,\nu_Y$ are projective bimeromorphic. Then we define $f^*\bd D = \overline{f^*\bd D_{\tilde X}}\in \bd{Div}(\tilde Y) = \bd{Div}(Y)$.

\subsection{Zariski Decomposition}

Let $X$ be a compact normal K\"ahler variety and $D$ be an $\mathbb R$-Cartier divisor on $X$. The following version of Zariski decomposition comes from \cite{9} and the formulation dues to \cite{5}.

\begin{df}
    A \textbf{Zariski decomposition} (in the sense of Fujita) of $D$ is a decomposition $\overline{D} = \mathbf P + \mathbf N$ where $\mathbf P,\mathbf N$ are b-divisors with real coefficients on $X$ such that
    \begin{enumerate}
        \item $\mathbf P\leq \overline{D}$ and $\mathbf P$ is b-nef;
        \item  If $\bd H\leq \overline{D}$ and $\textbf{H}$ is b-nef, then $\mathbf H\leq \mathbf P$.
    \end{enumerate}
  The b-divisor $\bd P = \bd P(D)$ is called the \textbf{semi-positive part}  and $\bd N = \bd N(D)$ is called the \textbf{negative part}.
\end{df}

From the definition, if a Zariski decomposition exists, then it is unique. If $D$ is a nef $\mathbb R$-divisor, then the Zariski decomposition exists and $\bd P(D) = \overline{D}$. 
Instead of studying existence in general, we will show the compatibility of the decomposition under pullbacks.

We first show in Corollary \ref{smooth-category} that we can reduce the discussion of Zariski decomposition to the category of (nonsingular) K\"ahler manifolds and projective bimeromorphic morphisms.

\begin{lem}
    Let $X$ be a compact normal K\"ahler variety and $\mathcal B'\subseteq \mathcal B_X$ be a filter basis. Then
    \begin{enumerate}
        \item $\bd D$ is b-nef in $\mathcal B_X$ if and only if $\bd D$ is b-nef in $\mathcal B'$;
        \item $\bd D_1\le \bd D_2$ in  $\mathcal B_X$ if and only if $\bd D_1\le \bd D_2$ in $\mathcal B'$;
        \item $D$ has a Zariski decomposition in $\mathcal B_X$ if and only if $D$ does in $\mathcal B'$. If the decomposition exists, we have $\bd P(D)_{\mathcal B_X} = \bd P(D)_{\mathcal B'}$.
    \end{enumerate}
\end{lem}

\begin{proof}
    (1) This follows from the fact that $D$ is nef if and only if $\pi^*D$ is nef.

    (2) It suffices to show $\bd E\ge 0$ in $\mathcal B_X$ implies $\bd E\ge 0$ in $\mathcal B'$. It follows from the fact that pushforward of an effective divisor is effective.

    (3) It suffices to show $\overline{D} = \bd P+\bd N$ is a Zariski decomposition in $\mathcal B_X$ if and only if $\overline{D} = \bd P+\bd N$ is a Zariski decomposition in $\mathcal B'$.

    In fact,  $\overline{D} = \bd P+\bd N$ is a Zariski decomposition if and only if $\bd P$ is b-nef in $\mathcal B_X$ and $\bd H\le \overline{D}$ implies $\bd H\le \bd P$ in $\mathcal B_X$. By (1) and (2), this is equivalent to $\bd P$ is b-nef in $\mathcal B'$ and $\bd H\le \overline{D}$ implies $\bd H\le \bd P$ on $\mathcal B'$, i.e. $\overline{D} = \bd P+\bd N$ is a Zariski decomposition. 
\end{proof}

\begin{cor}\label{smooth-category}
    Let $X, Y$ be compact normal K\"ahler varieties and $\pi: X'\to X$ be a projective bimeromorphic morphism. Let $D$ be an $\mathbb R$-Cartier divisor.  Then
    \begin{enumerate}
        \item  $D$ has a Zariski decomposition if and only if $\pi^*D$ does. If the decomposition exists, we have $\pi^*\bd P(D) = \bd P(\pi^*D)$.
        \item  Assume $X'$ is nonsingular. Then $D$ has a Zariski decomposition in $\mathcal B_X$ if and only if $\pi^*D$ has a Zariski decomposition in $\tilde{\mathcal B}_{\tilde X}$. If the decomposition exists, we have $\pi^*\bd P(D) = \bd P(\pi^*D)$.
    \end{enumerate}
\end{cor}

\begin{proof}
    For an $X'$ over $X$ or a projective resolution $\tilde X\to X$, applying the lemma to $\mathcal B_{X'}\subseteq \mathcal B_X$ and $\tilde{\mathcal B}_{\tilde X}\subseteq \mathcal B_X$.
\end{proof}

Next, we will study the pullback compatibility.
Let us first fix some terminology. Let $f:X\rightarrow Y$ be a surjective proper morphism between normal $\mathbb Q$-factorial varieties. Let $D$ be a divisor on $X$  and $E$ be a prime $\mathbb Q$-Cartier divisor on $Y$. We say $D$ \textbf{does not support the generic fiber} of $E$ if there is a prime divisor $F$ on $X$ such that $f(F) =  E$ but $F$ is not in $D$. We say a divisor $D$ is \textbf{unsaturated} over $Y$, if $D$ is vertical and $D$ does not support the generic fiber over any 
prime divisor $E$ of $Y$.

The main result of this section is:

\begin{prop}\label{pullback}
    Let $f:X\rightarrow Y$ be a proper contraction from a compact K\"ahler manifold $X$ to a projective manifold $Y$. Let $D$ be a $\mathbb R$-divisor on $Y$ and $E$ be an effective $\mathbb R$-divisor on $X$ that is unsaturated over $Y$.

    Then $f^*D + E$ has a Zariski decomposition if and only if $D$ does. Moreover, when the Zariski decomposition exists, we have $\bd P(f^*D + E) = f^*\bd P(D)$.
\end{prop}

In his paper \cite{9}, Fujita gives the definition of the decomposition and proves this proposition for projective varieties. We will extend this result to the K\"ahler case.

Let $f: X\to Y$ be a surjective proper morphism between $\mathbb Q$-factorial compact normal K\"ahler varieties. In the rest of this section, we will focus on the following two types of morphisms that are important for our main result (Theorem \ref{thmA}).

(I) projective morphism: $f$ is a projective bimeromorphic morphism from a compact K\"ahler manifold to a K\"ahler variety;

(II) projective base: $f$ is a proper surjective morphism from a compact K\"ahler manifold $X$ to a compact projective manifold $Y$.

We first show that unsaturated divisors won't contribute to b-nefness:

\begin{lem}\label{1.5}
     Let $f:X\to Y$ be a proper surjective morphism of type (I) or type (II). Let $D\geq 0$ be a vertical effective $\mathbb R$-divisor on $X$ and let $E_1,\dots,E_r$ be prime divisors on $X$. Let $Z$ be an irreducible component of $f(\operatorname{supp}(D))$. 
     
     Assume that 
    \begin{enumerate}
        \item $\operatorname{codim}_S(Z)\geq  2$; or 
        \item $\operatorname{codim}_S(Z) = 1$, each fiber of $f$ over a general point on $Z$ is connected, and $D$ does not support the generic fiber of $Z$.
    \end{enumerate}

     Then for any general point $y\in Z$, there is a subvariety $W\subseteq X_y$ satisfying: 
     \begin{enumerate}
         \item $[D]|_W$ is not pseudoeffective;
         \item $W$ is not contained in more than one $E_i$.
     \end{enumerate}
\end{lem}

This Lemma is the analytic version of \cite[Lemma 1.5]{9}, which is the key technique used in the proof of pullback compatibility of Zariski decomposition. In his paper \cite{9}, Fujita makes use of a hyperplane section argument. Since there is no ample divisor on a general analytic variety, we replace this with a Hodge theoretic argument:

\begin{lem}\label{D}
    Let $\pi:M\rightarrow S$ be a dominant morphism from a compact K\"ahler manifold $M$ to a projective variety $S$ where $\operatorname{dim}(M) = n$  and $\operatorname{dim}(S)\geq 2$. If $E$ is an effective divisor on $M$ which contracts to a point, then 
    \[
        \int_M[E]^2\wedge\omega^{n-2}<0
    \]
    for any K\"ahler form $\omega$ on $M$. 
\end{lem}

\begin{proof}
    Consider the symmetric $\mathbb R$-bilinear form on $H^{1,1}(M,\mathbb R)$:
    \[
       Q([\alpha],[\beta]) = \int_M\alpha\wedge{\beta}\wedge\omega^{n-2}.
    \]

    Let $H\subseteq S$ be a general hyperplane section which avoids the point $f(E)$. Then we have $Q(\pi^*[H],\pi^*[H]) > 0$. Let $V$ be the orthogonal complement of the linear subspace spanned by $\pi^*[H]$ in $H^{1,1}(X,\mathbb R)$. Then by the Hodge-Riemann relations, we have $Q$ is negative definite on $V$. Note that
    \[
        Q(\pi^*[H],[E]) = \int_M \pi^*[H]\wedge[E]\wedge\omega^{n-2} = 0
    \]
    We conclude that $[E]\in V$ hence $Q([E],[E]) < 0$. The statement follows.
\end{proof}

\begin{proof}[Proof of Lemma \ref{1.5}]
    By adding components of $D$ into the list $E_i$, we can assume $D\subseteq\cup_i E_i$. Assume $f$ is of type (I), i.e. $f$ is a projective bimeromorphic morphism. Since $f$ is bimeromorphic,  only case (1) is possible.  Then we can reduce to the projective case by the GAGA principal.

    For terminology and the general theory of the argument below, we refer to \cite[Section 2, 23 and 24]{lyu2026relativeminimalmodelprogram} for a comprehensive discussion.

    In fact, the statement is local on $Y$. Fix an affinoid $U = (\mathcal U,U)\subseteq (Y,Y)$. By the GAGA principle \cite[Appendix B, C]{GAGA}, we reduce to the case when $f: X\to Y$ is a projective birational morphism between excellent affine schemes. Next, by projectivity, shrinking $U$ if necessary, we decompose $f$ into $X\to \mathbb P^N_A\to Y = \operatorname{Spec}(A)$ where $A$ is an excellent ring of equal characteristic 0.

    By Bertini’s Theorem \cite[Theorem 10.1]{lyu2026relativeminimalmodelprogram}, after possibly shrinking $U$ again, we can find a general A-hyperplane section $H \subseteq Y$ such that $f^{-1}H$ is regular, $f^{-1}H$ meets each $E_i$ transversely, and $E_i|_{f^{-1}(H)}$ and $E_j|_{f^{-1}(H)}$ share no common component for $i\ne j$. Since $f$ is birational and $H$ is general, we have $f^{-1}H\to H$ is also birational. By repeating cutting $Z$ by the hyperplanes and shrinking $U$ if necessary, we reduce to the case when $f:X\to Y$ is a projective birational morphism between excellent surfaces and $Z$ is an isolated point.

    Then the statement is a consequence of the Hodge index Theorem. In fact, write $D = D_Z+D_R$ where $D_Z$ is the union of all those curves supported over $Z$. Then by the Hodge index theorem  for excellent surfaces \cite[Theorem 10.1]{MR3057950}, we have $D_Z^2<0$. As a result, there exists a curve $C$ supported over $Z$ such that $C.D = C.D_Z < 0$.  Since $C$ intersects $E_i$ properly if $E_i\ne C$, we know $C$ satisfies the property in the statement.

    Next, we will assume $f$ is of type (II). We will write $n = \dim(X)$.  We first assume that $Z$ is a point.
    
    If $\operatorname{dim}(Y)\ge 2$, then $Z$ satisfies assumption (1). In particular, we have $\operatorname{dim}(S)\ge 2$. Then, we fix a K\"ahler form $\omega$ on $X$ and consider the natural pairing $Q$ on $H^{1,1}(M,\mathbb R)$:
    \[
        [\alpha]\otimes[\beta]\mapsto Q([\alpha],[\beta]) := \int_X\alpha\wedge{\beta}\wedge\omega^{n-2}.
    \]
    Then,  write $D = D_Z+D_R$ where $D_Z$ are the components supported over the isolated point $Z$. By Lemma \ref{D}, we have 
    \[
        \int_X [D_Z]^2\wedge\omega^{n-2}<0.
    \]
    In particular, there exists an irreducible component $W$ of $D_Z$ such that
    \[
        \int_{W}[D]|_{W}\wedge\omega^{n-2} = \int_W[D_Z]|_W\wedge\omega^{n-2} =  \int_X[W]\wedge[D_Z]\wedge\omega^{n-2}<0.
    \]
    In particular, $[D]|_W$ is not pseudoeffective, and hence  this $W$ satisfies all the desired properties.

    Next, we will consider the case when $Y$ is a curve. This is essentially Zariski's lemma (Cf.\cite[Lemma III.8.2]{BPdV}). We will present a proof for completeness.
    
    Write $F := f^*(Z) = \sum_{i=1}^r m_iF_i$. Let $W = \oplus_i\mathbb R [F_i]$ be the $\mathbb R$-vector space spanned by $F_i$'s. We consider the bilinear form $Q:W\times W\to \mathbb R$:
    \[
        Q(\textbf{a},\textbf{b}) = \int_{X}(\sum a_i[F_i])\wedge(\sum b_j[F_j])\wedge\omega^{n-2}
    \]
    Let $N$ be the subspace of annihilators of $Q$. Then we have 
    \begin{enumerate}
        \item  $(m_1,\dots,m_r)\in N$;
        \item If $i\ne j$, then $-Q([F_i],[F_j])\le 0$;
        \item There's no nontrivial partition $I\cup J = [n]$ such that $-Q([F_i],[F_j]) = 0$ for $i\in I$ and $j\in J$. (This is because $F$ is connected.)
    \end{enumerate}
    Then, by \cite[Lemma I.2.10]{BPdV}, we deduce that $-Q\ge0$ and $N = \operatorname{span}((m_1,\dots,m_r))$. 
    
    Write $D = D_1+\dots+D_r$ such that $f(D_i) = p_i$ are distinct points and we assume $Z = p_1$. Since $D$ does not support the fiber $F$, we have $Q(D_1,D_1)<0$. Since $D_i$ are mutually disjoint, we have 
    \[
        \int_X [D]^2\wedge \omega^{n-2} = Q(D_1,D_1) + \sum_{i=2}^r\int_X[D_i]^2\wedge\omega^{n-2} < 0.
    \]

    Next, Since $Q(D_1,D_1)<0$, there is an irreducible component $W$ such that $Q(W,D_1)<0$. This implies that
    \[
        \int_{W}[D]|_{W}\wedge\omega^{n-2} = \int_X[W]\wedge[D]\wedge\omega^{n-2} = Q(W,D_1)<0.
    \]
   Once again, this implies that $[D]|_W$ is not pseudoeffective, and hence this $W$ satisfies all the desired properties.

    Now we consider the general case when $\operatorname{dim}(Z) = r>0$. Since $Y$ is assumed to be projective, following the argument in \cite[Lemma 1.5]{9}, we can reduce the statement to the case $\operatorname{dim}(Z) = 0$ by taking general hyperplane sections on $Y$:

    Let $S\subseteq Y$ be the intersection of $r$ general hyperplane sections $H_1,\dots,H_r\subseteq Y$. Then, we have $S$ intersects with $Z$ at a general point $u\in Z$ such that
    \begin{enumerate}
        \item $\dim(f^{-1}(S))+\dim(Z) = n$;
        \item $u$ is an isolated point of $S\cap Z$;
        \item $f^{-1}(S)$ is smooth at any point in $f^{-1}(u)$;
        \item $E_i|_{f^{-1}(S)}$ and $E_j|_{f^{-1}(S)}$ share no common component for $i\ne j$.
    \end{enumerate}
    Then we have the following commutative diagram
    \[
        \begin{tikzcd}
            T\arrow{d}{g}\arrow{r}{j} & X\arrow{d}{f}\\
            S\arrow{r} & Y
        \end{tikzcd}
    \]
    where $T$ is a resolution of $f^{-1}(S)$ that is isomorphic over $u$.
    
    By the case $\dim(Z) = 0$, there exists a $W\subseteq T$ supported over $u$ such that  $[j^*D]|_W$ is not pseudoeffective and $W$ is not contained in more than one $j^*E_i$. Since the resolution $T\to f^{-1}(S)$ is an isomorphism over $u$, we can think of $W\subseteq f^{-1}(S)\subseteq X$ and identify $[j^*D]$ with $[D|_{f^{-1}(S)}]$. This implies that $[D]|_{W} = [j^*D]|_W$ is not pseudoeffective and $W$ is not contained in more than one $E_i$.
\end{proof}

Then, the rest of the arguments in \cite{9} work if we replace \cite[Lemma 1.5]{9} by Lemma \ref{1.5}. Recall that by Theorem \ref{neftest}, a divisor $D$ is nef if and only if $[D]|_Z$ is pseudoeffective for all subvariety $Z\subseteq X$. Roughly speaking, we only need to replace each curve test `$D.C\ge 0$' in Fujita's argument by the subvariety test `$[D]|_Z$ is pseudoeffective'. As an example, we first show how to recover \cite[1.13]{9} using this replacement:

\begin{prop}\label{flatpullback}
    Let $f:X\to Y$ be a surjective contraction morphism from a compact K\"ahler manifold to a projective manifold. Let $L$ be a $\mathbb R$-divisor on $Y$ and $F\geq 0$ be an effective $\mathbb R$-divisor on $X$ such that $f^*L - F$ is nef. Then, for a flat model of $f$:
    $$\xymatrix{
    X'\ar[r]^\nu\ar[d]_{g}&X\ar[d]^f\\
    Y'\ar[r]^\pi& Y
    }$$
    there exists a $\mathbb R$-divisor $D\geq 0$ on $Y'$ such that $g^*D = \nu^*F$.
\end{prop}

\begin{proof}
    We first claim $\nu^*F$ is $g$-vertical. In fact, for any irreducible component $Z$ of a general fiber $X_{y'}'$,  we have $\displaystyle{\nu^*(f^*L - F)|_{Z} = -\nu^*F|_{Z}}$ is pseudoeffective. This means that $\nu^*F$ cannot intersect any irreducible component of a general fiber effectively. Thus $\nu^*F$ must be $g$-vertical.

    Take $D\geq 0$ to be the smallest $\mathbb R$-divisor on $Y'$ such that $R := g^*D-\nu^*F$ is effective.
    Suppose $R\neq 0$.  Taking a resolution if needed, we assume $X'$ is nonsingular. By Lemma \ref{1.5}, there is a subvariety $W$ contained in a fiber of $g$ such that $R|_W$ is not pseudoeffective. However, we also have
    \[
      R|_W = (g^*D-\nu^*F)|_W = -\nu^*F|_W = \nu^*(f^*L-F)|_W  
    \]
    is pseudoeffective because $\nu^*(f^*L-F)$ is nef. This leads to a contradiction.
\end{proof}

\begin{cor}\label{verticalnef}
    Let $f:X\to Y$ be a surjective contraction morphism from a compact K\"ahler manifold to a projective manifold. If $D$ is a $f$-vertical nef $\mathbb R$-divisor on $X$, then there is a b-nef $\mathbb R$-divisor $\bd E$ on $Y$ such that $\overline{D} = f^*\bd E$. 
\end{cor}

\begin{proof}
    Write $D = f^*L - F$ for some $L$ and $F\ge 0$. Then by the previous proposition, on the flat model $g$, there is a divisor $G$ on $Y'$ such that $g^*G = \nu^*F$. Let $\bd E = \overline{\pi^*L-G}$. We conclude that $\overline{D} = f^*\bd E$. Since $D$ is nef, we have $\bd E$ is b-nef.
\end{proof}

We will follow the terminologies in \cite{9}: Let $D$ be a $\mathbb R$-divisor on a complex manifold $X$. We say an effective  $\mathbb R$-divisor $E$ \textbf{clutches} $D$ if for any effective  $\mathbb R$-divisor $F$, we have $D-F$ is nef implies $F\geq E$.  We say $E$ is \textbf{numerically fixed} by $D$ if $\pi^*E$ clutches $\pi^*D$ for any projective bimeromorphic morphism $\pi:X'\rightarrow X$.

\begin{lem}\label{A}
    Let $f:X\rightarrow Y$ be a surjective morphism between compact K\"ahler manifolds.  If $f^*E$ clutches $f^*D$, then $E$ clutches  $D$.
\end{lem}

\begin{proof}
    Assume $F\geq 0$ such that $D-F$ nef. Then $f^*D-f^*F$ is nef. Since $f^*E$ clutches  $f^*D$, we have $f^*E\leq f^*F$ hence $E\leq F$.
\end{proof}

\begin{lem}\label{C}
    Let $f:X\rightarrow Y$ be a surjective contraction morphism of type (I) or (II). Let $E\geq 0$ be a vertical $\mathbb R$-divisor that is unsaturated over $X$.  Then $E$ is numerically fixed by $E + f^*D$ for any $\mathbb R$-divisor $D$ on $Y$.
\end{lem}

\begin{proof}
    The argument in \cite[Proposition 1.10]{9} works here if we replace \cite[Lemma 1.5]{9} by Lemma \ref{1.5}.   
\end{proof}

\begin{prop}\label{B}
    Let $f:X\rightarrow Y$ be a surjective contraction of type (I) or (II). Let $D$ be $\mathbb R$-divisor on $Y$ and $E$ be an effective $\mathbb R$-divisor on $Y$. Then $E$ is numerically fixed by $D$ implies $f^*E$ is numerical fixed by $f^*D$.
\end{prop}

\begin{proof}
    Suppose $F\geq 0$ and $f^*D-F$ nef. Then take a flat model:
    $$\xymatrix{
    V\ar[r]^\nu\ar[d]_g&X\ar[d]^f\\
    T\ar[r]^\pi&Y
    }$$
    By theorem \ref{flatpullback}, there is a $\mathbb R$-divisor $G$ on $T$ such that $\nu^*F = g^*G$. Then we have $g^*(\pi^*D-G) = \nu^*(f^*D-F)$ is nef. By Theorem \ref{pullbacknef}, we conclude that $\pi^*D-G$ is nef. Since $E$ is numerically fixed by $D$,  we have $\pi^*E$ clutches $\pi^*D$. This implies $G\geq \pi^*E$. This shows $\nu^*(F-f^*E) = g^*(G-\pi^*E)  \geq 0$ hence $F\geq f^*E$. 
\end{proof}

\begin{proof}[Proof of the Proposition \ref{pullback}]
  By Proposition \ref{C}, we know $E$ is always numerically fixed by $f^*D+E$. Thus we can assume $E = 0$.

  We first suppose $D$ has a Zariski decomposition. Then, by definition of Zariski decomposition, there is a projective bimeromorphic morphism $\pi: Y'\to Y$ such that $\pi^*D = H+N$ where $N$ is a effective divisor numerically fixed by $\pi^*D$ and $H$ is nef.  Then we consider the diagram
  $$\xymatrix{
  X'\ar[r]^\nu\ar[d]_{f'} & X\ar[d]^f\\
  Y'\ar[r]^\pi & Y
  }$$
  where $X'$ is a resolution of the main component of $X\times_YY'$. Then, by Proposition \ref{C}, we have $f'^*N$ is numerically fixed by $f'^*\pi^*D$ and $f'^*H$ is nef. In other word, $\nu^*f^*D = f'^*N + f'^*M$ is a Zariski decomposition.

  Conversely, we assume $f^*D$ has a Zariski decomposition. Replacing $X$ by a bimeromorphic model if necessary, we may write $f^*D = H + N$ where $H$ is nef and $N$ is effective and numerically fixed by $f^*D$. 
  
  Consider a flat model of $f$:
    $$\xymatrix{
    V\ar[r]^\nu\ar[d]_g&X\ar[d]^f\\
    T\ar[r]^\pi&Y
    }$$
  By Proposition \ref{flatpullback}, there is an effective divisor $N_T$ on $T$ such that $\nu^*N = g^*N_T$. Write $H_T = \pi^*D - N_T$. 
  
  We claim that $\pi^*D = H_T + N_T$ is a Zariski decomposition. In fact, since $\nu^*N = g^*N_T$ is numerically fixed by $\nu^*f^*D = g^*\pi^*D$, by Lemma \ref{A}, we have $N_T$ is numerically fixed by $\pi^*D$. Since $g^*H_T = \nu^*(f^*D-N)$ is nef, we have $H_T$ is nef as well.
\end{proof}

\section{Semi-stable Reduction in Codimension 1}

In this section, we will show that the semi-stable reduction in codimension one exists for morphisms from compact complex manifolds to projective manifolds. 

We will study branched coverings in the first subsection, and use them to present a proof for semi-stable reduction in codimension one (Theorem \ref{SSR}) in the third subsection.  

The proof of the semi-stable reduction in codimension one has 2 steps: we first make a suitable base change $Y'\to Y$ such that the semistable reduction can be constructed over a Zariski open subset of $Y'$ via toroidal blowing-ups. Next, we extend the the ideal sheaf of blowing up to the entire space. 

The second step is easy in the algebraic set-up, but in the analytic setting, the extension problem for coherent ideal sheaves can be subtle. We briefly discuss the extension problem in the second subsection and use the results to complete the second step.

We would also like to refer to \cite{Analytic-SSR} for semistable reduction of projective morphisms between analytic spaces.

\subsection{Branched Covering}

In this section, we will discuss branched coverings. We first classify branched coverings in local coordinates and then describe them globally in terms of toroidal morphism.

Let $(X,D)$ be a pair where $X$ is a (connected) complex manifold and $D$ is an snc divisor. A \textbf{branched covering} of $(X,D)$ is a surjective finite morphism $f:Y\to X$ from a normal analytic variety $Y$ that is \'etale over $X-D$.  We first classify branched coverings of $(\mathbb C^n,V(x_1\dots x_n))$. To be more precise, we have

\begin{thm}\label{branched-covering-local}
    There is a order-reversing correspondence between the poset of all isomorphic classes of branched coverings of $(\mathbb C^n,V(x_1\dots x_n))$ and the poset of all subgroups of finite indices of $\mathbb Z^n$.
\end{thm}

We will follow the conventions in \cite{Fulton-Toric_Varieties}. Let $N$ be a finitely generated free abelian group and $\Delta$ be a fan in $N_\mathbb R$. We will write $X(\Delta)$ as the toric variety induced by the fan $\Delta$.

\begin{lem}\label{Toric-Extension}
    Let $f:X\to Y$ be a morphism between toric varieties. If $f(T_X)\subseteq T_Y$ and the restriction $T_X\to T_Y$ is a homomorphism between complex tori, then $f$ is a toric morphism.
\end{lem}

\begin{proof}
    We have the following commutative diagram
    \[
        \begin{tikzcd}
            T_X\arrow{d}\arrow{r}{\phi} & T_Y\arrow{d}\\
            X\arrow{r}{f} & Y
        \end{tikzcd}
    \]
    Since all varieties are separated, morphisms agree on the dense open subset $T_X\times T_X\subseteq T_X\times X$ must agree everywhere. It follows that for any $t\in T_X$ and $x\in X$, we have $f(t\cdot x) = \phi(t)f(x)$.
\end{proof}

Let $N$ be $\mathbb Z^n$ and $\sigma = \operatorname{cone}(e_1,\dots,e_n)$. They induce a toric variety $X_{N,\sigma}=\mathbb C^n$ with $T_{N} = (\mathbb C^*)^n$. Then there is a canonical isomorphism
\[
     N \to \pi_1(T_{N},e)\quad\quad u\mapsto [\lambda^u|_{\mathbb S^1}]
\]
where $e \in T_{\sigma,N}$ is the distinguished point.

By topology, any subgroup $N'$ of $N$ of finite index induces a covering group $T_{N'}\to T_{N}$. This covering can be compactified to a toric covering between toric varieties:

\begin{lem}
    For any subgroup of finite index $N'\subseteq N$, the covering map $T_{N'}\to T_N$ extends to a unique surjective finite morphism $Y_{N'}\to X_{N,\sigma}$  (up to $X_{N,\sigma}$-isomorphisms) making the following diagram commutative
\[
    \begin{tikzcd}
        T_{N'}\arrow{r}\arrow{d} & T_{N}\arrow{d}\\
        Y_{N'}\arrow{r} & X_{N,\sigma}
    \end{tikzcd}
\]
    Furthermore, for nested subgroups $N''\subseteq N'\subseteq N$, there exists a unique commutative triangle
    \[
        \begin{tikzcd}
            Y_{N''}\arrow{r}\arrow{rd} & Y_{N'}\arrow{d}\\
            & X_{N,\sigma}
        \end{tikzcd}
    \]
    that extends the tower of covering spaces $T_{N''}\to T_{N'}\to T_{N}$.
\end{lem}

\begin{proof}
    The uniqueness follows immediately from Proposition \ref{finite-cover-uniqueness}. Set $Y_N = X_{N',\sigma}$, the toric variety induced by the same cone but in the finer lattice $N'$. Then $X_{N',\sigma}\to X_{N,\sigma}$ is a such a morphism which shows the existence. 

    Next, let $N''\subseteq N'\subseteq N$ be nested subgroups.  Since $X_{N,\sigma},Y_{N'},Y_{N''}$ are separated, if the commutative triangle exists, it must be unique. 
    
    To show existence, since the toric varieties $Y_{N''},Y_{N'},X_{N,\sigma}$ are given by the same cone $\sigma = \operatorname{cone}(e_1,\dots,e_n)$ in $N'',N',N$ respectively, the cone structures induce a triangle of toric morphisms
    \[
        \begin{tikzcd}
            Y_{N''}\arrow{r}\arrow{rd} & Y_{N'}\arrow{d}\\
            & X_{N,\sigma}
        \end{tikzcd}
    \]
    which extends the triangle of covering spaces $T_{N''}\to T_{N'}\to T_{N}$. 
\end{proof}

The previous lemma shows that there is a correspondence $N'\mapsto Y_{N'}$ which assigns each subgroup $N'\subseteq N$ an isomorphic class of surjective finite coverings $Y_{N'}\to \mathbb C^n$. Then, Theorem \ref{branched-covering-local} follows from the following proposition:

\begin{prop}\label{toric-covering}
    The assigment $N'\mapsto [Y_{N'}\to \mathbb C^n]$ is an order-reversing correspondence between the poset of all isomorphic classes of branched coverings of $(\mathbb C^n,V(x_1\dots x_n))$ and the poset of all subgroups of finite indices of $\mathbb Z^n$.

    In addition, if we endow each $T_{N'}$ with the natural multiplicative group structure, then each $Y_{N'}$ has a unique $T_{N'}$-toric structure. Furthermore, all morphisms $Y_{N''}\to Y_{N'}$ induced by $N''\subseteq N'$ are toric morphisms.
\end{prop}

\begin{proof}
    Let $\pi:Y\to (\mathbb C^n,V(x_1\dots x_n))$ be a branched covering. Then from the definition of branched covering, the restriction $\pi^{-1}(Y_N)\to Y_N$ is a covering map, which must be of the form $T_{N'}\to T_N$. By the previous lemma, we know $\pi$ is of the form $Y_{N'}\to X_{N',\sigma}$. The second statement of the previous lemma shows that this is a correspondence between posets.

    Next, from the construction of $Y_{N'}$, we have $Y_{N'}$ can be seen as a $T_{N'}$-toric variety.  Since $Y_{N'}$ is separated, this toric structure is unique when the group structure on $T_{N'}$ is given. Finally, by Lemma \ref{Toric-Extension}, the morphism $Y_{N''}\to Y_{N'}$ induced by $N''\subseteq N'$ is toric.
\end{proof}

In practice, we may want to apply the theory to the open polydisc $\mathbb D^n$. This can be reduced to the $\mathbb C^n$ case: If $X = \mathbb D^n$ with $\Delta = V(x_1,\dots,x_n)$ and $\nu: X'\to X$ is a finite morphism from a normal analytic variety $X'$ such that $\nu$ is \'etale over $X-\Delta$. Then there is a toric morphism $X_{N',\sigma}\to \mathbb C^n$ making the following diagram commutative:
\[
    \begin{tikzcd}
        X'\arrow{r}{\nu}\arrow{d} & \mathbb D^n\arrow{d}\\
        X_{N',\sigma}\arrow{r}  & \mathbb C^n
    \end{tikzcd}
\]
where the vertical morphisms are open immersions.

Next, we will use the language of toroidal geometry to describe branched coverings.
We will follow the conventions in \cite{Analytic-SSR}. 

\begin{df}
    A \textbf{toroidal pair} is a pair $(X,\Delta)$ where $X$ is a normal analytic variety and $\Delta$ is a reduced divisor such that for any $x\in X$, there is an affine toric variety $X_\sigma$ associated to a strongly convex rational polyhedral cone $\sigma$ with the open dense torus $T_\sigma$, a distinguished point $t\in X_\sigma$, and an isomorphism
\[
\widehat{\mathscr{O}}_{X,x} \cong \widehat{\mathscr{O}}_{X_\sigma,t}
\]
of complete local $\mathbb{C}$-algebras which sends the reduced ideal sheaf of $X - U$ in $\widehat{\mathscr{O}}_{X,x}$ to that of $X_\sigma - T_\sigma$ in $\widehat{\mathscr{O}}_{X_\sigma,t}$. We call the germ $t \in X_\sigma$ a \textbf{local model at $x$}. A toroidal pair $(X,\Delta)$ is \textbf{strict}  if each irreducible component of $\Delta$ is normal.
\end{df}

\begin{df}
    Let $(X, \Delta_X)$ and $(Y, \Delta_Y)$ be toroidal pairs. Then a dominant morphism $f \colon X \to Y$ is \textbf{toroidal} if for any point $x \in X$, there exist a local model $s \in X_\sigma$ at $x$, a local model $t \in X_\tau$ at $f(x)$, and a toric morphism $g \colon X_\sigma \to X_\tau$ which maps $s$ to $t$ such that the formal isomorphisms $(x \in X) \cong (s \in X_\sigma)$ and $(f(x) \in Y) \cong (t \in X_\tau)$ are compatible with $f$ and $g$.
\end{df}

\begin{remark}
    We are following the conventions in \cite{Analytic-SSR}.
    In particular,  $(X,\dt)$ is a toroidal pair if $X-\dt\subseteq X$ is a toroidal embedding in the sense of \cite[page 54]{Kempf1973}. Also, $(X,\dt)$ is a strict toroidal pair if $X-\dt\subseteq X$ is a toroidal embedding without self-intersection in the sense of \cite[page 57]{Kempf1973}.
\end{remark}

\begin{lem}\label{cover-strict-toroidal}
    Let $\pi: Y\to (X,\dt)$ be a branched covering over a log smooth variety $(X,\dt)$. Then $(Y,\pi^{-1}\dt)$ is a strict toroidal pair and $\pi$ is a toroidal morphism.
\end{lem}

\begin{proof}
    Since $Y-\pi^{-1}\dt\to X-\dt$ is \'etale, by  \cite[Proposition 2.25]{Analytic-SSR}, it extends to finite torodial morphism $Z\to X$ where $Z$ is normal and $Y-\pi^{-1}\dt\subseteq Z$ is a toroidal embedding. By Proposition \ref{finite-cover-uniqueness}, we have $Z\cong Y$ and the isomorphism respects the common open subset $Y - \pi^{-1}\dt$. This shows that $(Y,\pi^{-1}\dt)$ is a toroidal pair and $\pi$ is a toroidal morphism.

    To see $(Y,\pi^{-1}\dt)$ is strict, we fix $y\in Y$ and write $x = \pi(y)\in X$. Then locally near $y$ and $x$, we have $\pi$ restricts to a branched covering of polydiscs $B'\to (B,\dt_B)$, where $\dt_{B} = V(x_1\dots x_n) =: \sum D_i$ and $\dt\cap B \subseteq\dt_{B}$. 

    By the local description of branched coverings, $\pi$ is a restriction of a toric morphism $\bar\pi:X_{N',\sigma}\to \mathbb C^n$ where  $N'$ is a free subgroup of $N$ with finite index and $\sigma = \operatorname{cone}(e_1,\dots,e_n)\subseteq N_\mathbb R = N'_\mathbb R$. Since $X_{N',\sigma}$ is a toric variety induced by a fan, all the invariant divisors of $X_{N',\sigma}$ are normal. For each divisor $D_i$ in $\mathbb C^n$ corresponding to $e_i$, we have $\bar\pi^{-1}(D_i)$ corresponds to the same ray in $N'_\mathbb R$, which is in particular irreducible.

    We conclude that for each $D_i$ in $B$, the inverse image $\pi^{-1}(D_i)$ is irreducible and normal. This implies that every irreducible component of $\pi^{-1}\dt$ admits a Euclidean cover in which it is both normal and irreducible. As a result, each irreducible component of $\pi^{-1}\dt$ is normal. 
\end{proof}

\subsection{Extension of Ideal Sheaves}

Let $(X,U,\mathscr I)$ be a triple where $X$ is an analytic variety, $U$ is a Zariski open and $\mathscr I$ is an ideal sheaf on $U$. We are interested in when $\mathscr I$ can be extended to $X$.

Let $X$ be an analytic variety and $Z\subseteq X$ be an analytic subset.
Let $\mathscr I$ be an ideal sheaf. We say that $\mathscr I$ is \textbf{cosupported} on $Z$ if $\mathscr I|_{X-Z} = \mathscr O_{X-Z}$.

\begin{lem}\label{normal-extension}
    Let $X$ be a compact normal analytic variety and $Z,F\subseteq X$ be proper analytic subsets. Let $\mathscr I$ be an ideal sheaf on $X-Z$ that is cosupported on $(X - Z)\cap F$. Then $\mathscr I$ extends to an ideal sheaf $i_*\mathscr I$ on $X$ where $i:X-Z\to X$ is the open immersion.
\end{lem}

\begin{proof}
    Replacing $F$ by $\overline{F-Z}$ if necessary, we can assume $F$ and $Z$ share no common irreducible components. Write $E = F\cap Z$. Then $E$ has  codimension at least $2$.

    We first claim that $\mathscr I$ extends to $X - E$. In fact, cover $X-E$  by two open subsets $X-Z$ and $X-F$. Noticing that the ideal sheaf $\mathscr I$ on $X-Z$ and the trivial ideal sheaf $\mathscr O_{X-F}$ on $X-F$ agree in the overlap, they glue together to an ideal sheaf $\mathscr J$ on $X - E$ that extends $\mathscr I$.
    
    Since $X$ is normal and $E$ has codimension at least 2, by Hartog extension, we have $i_*\mathscr O_{X-Z} = \mathscr O_X$. As a result, $\mathscr J$ further extends to $X$.
\end{proof}

Unlike that in the algebraic category, the coherence of the extending ideal sheaf  is not for free. We will focus on the case when the ideal sheaf is algebraic.

We say a triple $(X,U,\mathscr I)$ is \textbf{locally algebraic} if for every $x\in X$, there exists a triple $(Y,V,\mathscr J)$ where $Y$ is an algebraic variety, $V\subseteq Y$ is a Zariski open subset and $\mathscr J$ is a coherent ideal sheaf on $V$ such that the analyticization $(Y^{an},V^{an},\mathscr J^{an})$ is isomorphic to $(X,U,\mathscr I)$ near $x$.

\begin{lem}\label{coherent-extension}
    Let $(X,U,\mathscr I)$ be a locally algebraic triple such that $X$ is normal and $\mathscr I$ is coherent in $U$ cosupported in a proper analytic subset. Then $\mathscr I$ extends to a coherent ideal sheaf on $X$.
\end{lem}

\begin{proof}
    Let $i:U\to X$ be the open immersion. By Lemma \ref{normal-extension}, we know $i_*\mathscr I$ is an ideal sheaf that extends $\mathscr I$. Since for algebraic varieties, extensions of coherent ideal sheaves are coherent, we know for any $x\in X$, the ideal sheaf $\mathscr I$ is coherent near $x$. Since coherence is a local property, we know $\mathscr I$ is coherent.
\end{proof}

In our application, we are interested in the extensions of ideal sheaves coming from toroidal blowing-ups. These ideals are locally algebraic, because they come from toroidal strata.

\subsection{Semi-stable Reduction in codimension one}

In this section, we will show that semi-stable reduction in codimension one (See \cite[Theorem 4.3]{8} for the algebraic case) exists for morphism from complex manifolds to projective morphisms.

\begin{thm}[Semi-stable reduction in codimension one]\label{SSR}
Let $f: X \to Y$ be a surjective morphism from a compact complex manifold to a smooth projective variety. Let $\Sigma_Y$ be an snc divisors on $Y$ such that $f$ is smooth away from $\Sigma_Y$ and $\Sigma_X = f^{-1}(\Sigma_Y)$ is an snc divisor. 

Then, there exists an integer $N>0$ such that: if $\pi: Y' \to Y$ is a finite covering from a smooth projective variety $Y'$ such that $\Sigma_{Y'} := \pi^{-1}(\Sigma_Y)$ is an snc divisor, and $N$ divides the ramification indices of $\pi$ over the prime components of $\Sigma_{Y'}$, and $\pi$ is \'etale away from an snc divisor $B$  such that $\Sigma_Y\subseteq B$ and $f^{-1}B$ is an snc divisor, then, there exists a commutative diagram
    \[
    \begin{tikzcd}
        X\arrow[swap]{d}{f} & X\times_YY'\arrow{d}\arrow{l} & X'\arrow{l}{p}\arrow{dl}{f'}\\
        Y & Y'\arrow[swap]{l}{\pi}
    \end{tikzcd}
    \]
with the following properties:
\begin{enumerate}[(a)]
  \item $X'$ is nonsingular and $\Sigma_{X'} := \pi'^{-1}(\Sigma_X)$ is an snc divisor, where $\pi': X' \to X$ is the composition;
  \item $p$ is a projective morphism which is an isomorphism over $Y'-\Sigma_{Y'}$. In particular, $f'$ is smooth over $Y' - \Sigma_{Y'}$;
  \item $f'$ is semi-stable in codimension one: there exists a big Zariski open $U\subseteq Y'$ such that for any prime divisor $D\subseteq Y'$, we have $f'^*(D|_U)$ is an snc divisor.
\end{enumerate}
\end{thm}

\begin{proof}
    \noindent\textbf{Step 1:} We will show that there is a Zariski open subset $Y_0'$ of $Y'$ over which the semi-stable reduction exists.

    Write $\Sigma_X = \sum D_i + \sum E_i$ and $\Sigma_Y = \sum F_i$ as  sums of distinct irreducible components in a way that
    \begin{enumerate}
        \item $f(D_i) = F_i$ for some $F_i$;
        \item $f(E_i)\subsetneqq F_i$  for some $F_i$.
    \end{enumerate}
    Then $f^*\Sigma_Y  = \sum_i m_iD_i+\sum_in_iE_i$. We set $N_0$ as the greatest common multiple of all $m_i$'s.

    Let $\pi: Y'\to Y$ be a finite cover as in the statement. We write $B = W\cup\Sigma_{Y'}$ such that $W$ and $\Sigma_{Y'}$ have no common irreducible components. Let $\hat X'$ be the normalization of $X\times_YY'$:
     \[
        \begin{tikzcd}
  & & \hat{X}'\arrow[bend right=20]{dll}[swap]{\nu'}\arrow[bend left=20]{ddl}{\hat f'}\arrow{dl}\\
  X \dar{f} & \lar X \times_Y Y' \dar\\
  Y & \lar{\pi} Y'
\end{tikzcd}
    \]
    We denote $\hat \nu:\hat X'\to X$ as the composition, and $\Sigma_{Y'} = (\pi^{-1}\Sigma_Y)_{red} = \sum F_j'$ and $\Sigma_{\hat X'} = (\hat \nu^{-1}\Sigma_X)_{red} = \sum D_j' + \sum E_j'$.  Then Lemma \ref{cover-strict-toroidal} implies that $(\hat X',\hat f'^{-1}(W'+\Sigma_{Y'}))\to (X,f^{-1}(W+\Sigma_Y))$ is a finite toroidal morphism between strict toroidal pairs.
    
    Then,  since $\pi$ has divisible enough ramification index along each irreducible component of $\Sigma_Y$, we have $\hat f'$ satisfies condition (a). To achieve condition (a)(b), it suffices to find a blowing-up of $\hat X'$ over $\Sigma_{Y'}$ that does not increase the multiplicity of each $D_j'$. 

    Let $Y_0 = Y - f(\sum E_i)\cup S\cup W$ where $S = \cup_{i\neq j} F_i\cap F_j$, and $X_0, \hat X'_0,Y'_0$ be the preimages. We will still denote the restriction by $\hat f': \hat X_0'\to Y'_0$. 

    Consider the triple $(\hat X'_0,\hat X'_0-\hat f'^*(\Sigma_{Y'_0})_{red},\hat f'^*(\Sigma_{Y'_0}))$.  
    Since $\pi$ has divisible enough ramification indices along each $F_i$, we know $\hat f'^*(\Sigma_{Y'_0})  = \sum D_j'$ is a reduced Cartier divisor whose support is exactly $\hat f'^*(\Sigma_{Y'_0})_{red}$.

   By \cite[section 2.3]{Kempf1973}, there is an integer $N$ depending on $(f,\Sigma_Y)$ such that: if $N$ divides each of the ramification indices of $\pi$, then there is an ideal sheaf $\mathscr I_0$ supported over $\Sigma_{Y'_0}$ satisfying
   \begin{enumerate}
        \item $\mathscr I_0$ comes from a projective subdivision of the polyhedral complex induced by $(\hat {X}'_0,\hat X'_0-f'^*(\Sigma_{Y'_0})_{red})$;
       \item $p_0:Z_0 = \operatorname{Bl}_{\mathscr I_0}(\hat X'_0)\to \hat X'_0$ is a resolution of singularities;
       \item $p_0^{*}(\sum D_i')$ is a reduced snc divisor.
   \end{enumerate}

    \noindent\textbf{Step 2:} We will extend the semi-stable reduction over $Y_0'$ to the entire $Y'$. Note that we only expect the extension to be semi-stable only in codimension one. We are not trying to control the multiplicities of $E_i$'s.

    Since $\hat X'$ is normal and $\mathscr I_0$ is an ideal sheaf cosupported in a divisor of $\hat X'$, by Lemma \ref{normal-extension}, $\mathscr I_0$ extends to an ideal $\mathscr I$ on $\hat X'$. Since $(\hat X',\hat f'^{-1}(W'+\Sigma_{Y'}))$ is a toroidal pair, we have $(\hat X',\hat X_0')$ is locally isomorphic to a pair $(X_\sigma,X_\sigma^\circ)$ where $X_\sigma$ is a toric variety and $X_\sigma^\circ$ is a Zariski open obtained by removing some toric divisors. By \cite[page 83 Lemma 3]{Kempf1973},  there exists a finite open cover $\set{V_i}$ of $\hat X_0'$ such that $\mathscr I_0|_{V_i} = \sum\mathscr O(m_j D_j'|_{V_i})$, where $D_i'$ are codimension one strata.  Consequently, the triple $(\hat X',\hat X'_0,\mathscr I_0)$ is locally algebraic, and by Lemma \ref{coherent-extension}, we conclude that $\mathscr I = i_*\mathscr I_0$ is coherent.

   In conclusion, there is an coherent ideal sheaf $\mathscr I$ on $\hat X'$ supported over $\Sigma_{Y'}$ satisfying
   \begin{enumerate}
       \item $Z = \operatorname{Bl}_{\mathscr I}(\hat X')$ is nonsingular over $\hat X_0$;
       \item $p^{*}(\sum D_i')$ is a reduced snc divisor over $\hat X'_0$ where $p_0$ is the blowing-up morphism.
   \end{enumerate}
   From the construction, (a-c) are satisfied away from a codimension $\geq 2$ closed subset in $\Sigma_{Y'}$. By blowing up centers sitting over this codimension $2$ subset, we obtain a morphism $f':X'\to Y'$ that satisfies all desired properties (a-c).
\end{proof}

Since we are interested in the case when the base $Y$ is projective, by Kawamata's covering trick, the base change in the statement of Theorem \ref{SSR} exists (See \cite[Theorem 17]{11} and \cite[Theorem 4.1, Remark 4.2]{8}).

\section{The Canonical Bundle Formula}

In this section, we will show that the canonical bundle formula holds for morphisms from K\"ahler varieties to projective varieties (Theorem \ref{moduli_b-nef}).

The canonical bundle formula is an important tool in birational geometry developed in \cite{MR700593}\cite{9}\cite{MR1863025}\cite{8}\cite{6}\cite{ambro2022positivitymoduli}. In \cite{hacon2024canonicalbundleformulaadjunction}, a canonical bundle formula for projective morphisms between analytic varieties is established.

The b-nefness of the moduli part is established in the fundamental paper \cite{8} for projective varieties. To extend the theorem to K\"ahler domains, we  study the roles that the domain and base play in the proof
.

Our observation is that morphisms from K\"ahler varieties to projective varieties provide a suitable setup for the proof to work. In fact, in the proof, the domain is only used to produce semi-positivity while the base is where various base changes are made. 

On the one hand, the semi-positivity comes from an analytic estimate which still holds for K\"ahler manifolds. On the other hand, the base changes techniques, which heavily rely on hyperplane sections, are still available since the base is assumed to be projective. Another thing to notice is that the subtlety of analytic nefness won't appear in this set-up, since the moduli divisor is on the projective base.

\subsection{klt-trivial Fibrations}

Kodaira's canonical bundle formula provides an important tool to compare the canonical bundles along an elliptic fibration. 
In this section, we will discuss Kawamata's canonical bundle formula, which is a higher dimensional analogue of Kodaira's canonical bundle formula.

We refer to \cite{8} for detailed discussions of the basic concepts. Following \cite{8}, we first fix some terminologies.
A \textbf{log pair}  $(X,B)$ is a normal analytic variety $X$ with a $\mathbb Q$-Weil divisor $B$ such that $K_X+B$ is a $\mathbb Q$-Cartier divisor. We call a log pair $(X,B)$ a \textbf{log variety} if $B\geq 0$. The discrepancy b-divisor is 
\[
    \bd A(X,B) = \bd K - \overline{K_X+B}.
\]
where $\bd K$ is the b-canonical divisor.

\begin{df}
    A \textbf{klt-trivial fibration}  $f:(X,B)\to Y$ consists of a pair $(X,B)$ and a surjective contraction morphism from a compact normal variety $X$ to a smooth variety $Y$ such that
    \begin{enumerate}
        \item  $(X,B)$ is klt over a dense Zariski open of $Y$;
        \item  $\operatorname{rank}f_*\mathscr O_X(\lceil\bd A(X,B)\rceil) = 1$;
        \item  $K_X+B\sim_{f,\mathbb Q}0$.
    \end{enumerate}
\end{df}

\begin{remark}
    Note that in \cite{8}, it is called a lc-trivial fibration. 
\end{remark}

In this article, we will focus on the case when $Y$ is a projective variety. Let $f: (X,B)\to Y$ be a fibration satisfying (1) and (3). For any prime divisor $P\subseteq Y$, let $U$ be a Zariski open of $Y$ meeting $P$ such that $P|_U$ is Cartier. Then there exists a $t\in \mathbb R$ such that $(X_U,B|_{X_U}+tf^{-1}(P|_{U}))$ is lc. Since such $t$'s are bounded above, we define $b_P\in\mathbb R$ such that $1-b_P$ is the supremum of all those $t$'s, i.e. 
\[
    1-b_P = \operatorname{lct}(X_U,B|_{X_U},f^{-1}(P|_{U})).
\]
We define the \textbf{discriminant divisor} as  $B_Y = \sum_{P} b_P P$.

Note that if $\mu:(X',B')\to (X,B)$ is a bimeromorphic model such that $K_{X'}+B' = \mu^*(K_X+B)$, then we have $(X',B_{X'})\to Y$ also satisfies (1)(3) and induces the same  discriminant divisor $B_Y$ as $(X,B)\to Y$ does.

\begin{exm}\label{ssr-discriminant}
    Let $f:(X,B)\to Y$ be a fibration satisfying (1) that is semi-stable in codimension one. Let $P$ be a prime divisor on $Y$. Since $b_P$ is defined in a neighborhood of $P$, we have $B+tf^*P$ has snc support over some neighborhood $U$ of $P$. Let $\sum b_iE_i$ be the sum of components in $B$ such that $f(E_i) = P$. Then  we have $b_P = \max(b_i)$.
\end{exm}

Once a meromorphic function $\phi\in\mathbb C(X)$ has been fixed,  there exists a unique $\mathbb Q$-divisor $M_Y$, called the moduli divisor, such that
\[
    K_X + B + \frac{1}{r}(\phi) = f^*(K_Y + B_Y + M_Y).
\]

Let $f:(X,B)\to Y$ be a fibration satisfying (1)(3) and $\sigma: Y'\to Y$ be a birational base change between projective varieties. A \textbf{fibration induced by $\sigma$} is a fibration $f':(X',B')\to Y'$ where $X'$ is a projective resolution of the main component of $X\times_YY'$ and $(X',B')$ is the crepant log structure on $X'$: 
\[
\xymatrix{
X'\ar[r]\ar[rd]_{f'} & X\times_YY'\ar[r]\ar[d] & X\ar[d]^f\\
& Y'\ar[r]^\sigma & Y
}
\]

Let $B_{Y'}$ be the discriminant divisor of the induced fibration $f'$. Then we have $\sigma_*(B_{Y'}) = B_Y$.  Thus the fibration $f$ induces a $\mathbb Q$-b-divisors $\textbf{B}$ on $Y$, called the \textbf{discriminant b-divisor}. Similarly, we can define the \textbf{moduli b-divisor} $\bd M$.

Then, we have the canonical bundle formula:

\begin{thm}\label{moduli_b-nef}
    Let $f:(X,B)\to Y$ be a klt-trivial fibration from a K\"ahler variety $X$ to a projective variety $Y$. Then
    \begin{enumerate}
        \item $\bd K+\bd B$ is $\mathbb Q$-b-Cartier;
        \item $\bd M$ is b-nef.
    \end{enumerate}
\end{thm}

This thoerem is stated and proved by Florin Ambro in \cite{8} for projective varieties.
In the next subsection, we will adapt Ambro's idea and extend the set-up to K\"ahler domains.

\subsection{Proof of the canonical bundle formula}

For a surjective morphism $f:X\to Y$ and a divisor $D\subseteq X$. We denote $D^h$ (resp. $D^v$) as the horizontal (resp. vertical) parts of $D$.

Let $f:(X,B)\to Y$ be a  klt-trivial fibration and $\Sigma_X,\Sigma_Y$ be reduced divisors on $X,Y$. Following  \cite[Definition 8.3.6]{6}, we say $(f,B,\Sigma_X,\Sigma_Y)$ satisfies \textbf{standard normal crossing assumptions}  if  
\begin{enumerate}
    \item $\Sigma_X,\Sigma_Y$ support $B$ and $B_V,M_V$ respectively;
    \item $(f^{-1}\Sigma_Y)_{red}\subseteq \Sigma_X$ and  $f(\Sigma^v_X)\subseteq \Sigma_Y$;
    \item $f$ is smooth away from $\Sigma_Y$;
    \item $(X,\Sigma_X)$ and $(Y,\Sigma_Y)$ are log smooth and $\Sigma_X^h$ is relative snc over $Y$.
\end{enumerate}

\begin{lem}\label{standard-normal-crossing}
    Let $f:(X,B)\to Y$ be a fibration satisfying (1)(3) in the definition of the klt-trivial fibration. Then there exists a birational morphism $\sigma:Y'\to Y$ and a commutative diagram
    \[
        \begin{tikzcd}
            (X',B')\arrow{r}{\mu}\arrow{d}{f'} & (X,B)\arrow{d}{f}\\
            Y'\arrow{r}{\sigma} & Y
        \end{tikzcd}
    \]
    where $f':(X',B')\to Y'$ is a fibration induced by $\sigma$, such that there exist reduced divisors $\Sigma_{X'}\subseteq X'$ and $\Sigma_{Y'}\subseteq Y'$ such that $(f',B',\Sigma_{X'},\Sigma_{Y'})$ satisfies standard normal crossing assumptions.
\end{lem}

\begin{proof}
    As in the previous discussion, any birational base change induces a fibration. After replacing $f$ by a fibration induced by a resolution of $Y$, we could assume $(X,B)$ is log smooth and $Y$ is smooth.
    
    Let $\Sigma_Y$ be a reduced divisor containing $B_Y,M_Y,f(B^v)$, the locus where $f$ is not smooth and the locus of points over which $B$ is not relative snc. Let $\Sigma_X = f^{-1}(\Sigma_Y)_{red}\cup B_{red}$.
    Then we have $f$ satisfies (1-3). Let $\sigma: Y'\to (Y,\Sigma_Y)$ be a log resolution given by blowing up centers inside $\Sigma_Y$. By blowing up the main component of $X\times_YY'$ along centers over $\Sigma_Y$, we obtain a fibration $f':(X',B')\to Y'$ that satisfies (4). 
\end{proof}

In order to apply Hodge theory, we need to construct a finite base change to clean the denominators. To define this auxiliary fibration, we need to understand the behaviors of $B_Y, M_Y$ under finite base changes:

\begin{prop}\cite[Lemma 5.1]{8}\label{finite-base-change}
    Consider the commutative diagram of normal varieties
    \[
        \begin{tikzcd}
            (X',B')\arrow{d}{f'}\arrow{r}{\nu} & (X,B)\arrow{d}{f}\\
            Y'\arrow{r}{\tau} & Y
        \end{tikzcd}
    \]
    such that
    \begin{enumerate}
        \item $(X,B)$ is lc over some dense Zariski open of $Y$;
        \item $\tau$ is finite and $\nu$ is generically finite, and $f,f'$ are proper;
        \item $\nu^*(K_X+B) = K_{X'}+B_{X'}$.
    \end{enumerate}
    Then $\tau^*(K_Y+B_Y) = K_{Y'}+B_{Y'}$ as finite pullback of Weil divisors.
\end{prop}

\begin{proof}
    This is proved in \cite[Theorem 3.2]{MR2698988}. The proof still works for analytic varieties.
\end{proof}

\begin{lem}\label{auxillary-base-change}
    Let $(f,B,\Sigma_X,\Sigma_Y)$ be a fibration that satisfies standard normal crossing assumptions. Then there exists projective birational base change $\sigma: Y'\to Y$ and a commutative diagram
    \[
    \begin{tikzcd}
        (X',B')\arrow[swap]{d}{f'} & (V',B_{V'})\arrow{ld}{g'}\arrow[swap,dashed]{l}{h'}\\
        Y'
    \end{tikzcd}
    \]
    where $h'$ is a generically finite proper map, $K_{V'}+B_{V'} = g'^*(K_{X'}+B')$ and, $\Sigma_{X'}$ is a reduced divisor on $X'$ and  $\Sigma_V$ is a reduced divisor on $V$ such that the following properties hold:
    \begin{enumerate}
        \item $(h',B_{V'},\Sigma_{V'},\Sigma_{Y'})$ satisfies standard normal crossing assumptions;
        \item $K_{V'}+B_{V'}+(\psi) = h^*(K_{Y'}+B_{Y'}+M_{Y'})$ where $\psi$ is a principal divisor.
    \end{enumerate}
\end{lem}

\begin{proof}
     Assume $K_V+B_V+\frac{1}{r}(\phi) = h^*(K_Y+B_Y+M_Y)$. Let $\tilde X\to X$ be the normal cyclic covering induced by $\sqrt[r]{\phi}$ and write  $\Sigma_{\tilde X} = \pi^*(\Sigma_X)_{red}$. Then there exists a meromorphic function $\psi$ such that $\psi^r = \pi^*\phi$ so that condition (2) holds. Let $d:V\to \tilde X$ be a resolution by blowing up along centers over $\Sigma_X$ and write $\Sigma_{V'} = (d^{-1}\Sigma_V)_{red}$. Then by Proposition \ref{finite-base-change}, we have $(V,B_V)\to Y$ satisfies standard normal assumptions (1)(2).
     
     Let $\Sigma\subseteq Y$ be the union of $\Sigma_Y$ and the locus where $(V',B_{V'})\to Y'$ is not relatively log smooth. After taking a log resolution $\sigma: (Y',\Sigma_{Y'})\to (Y,\Sigma)$, one can blow up $V'$ and $X'$ respectively over $\Sigma_{Y'}$ to achieve standard normal assumptions (3)(4).
\end{proof}

For a klt-trivial fibration $(f,B,\Sigma_X,\Sigma_Y)$ that satisfies standard normal crossing assumptions, we call the triangle
\[
    \begin{tikzcd}
        (X,B)\arrow[swap]{d}{f} & (V,B_V)\arrow{ld}{g}\arrow[swap,dashed]{l}{h}\\
        Y
    \end{tikzcd}
\]
in the previous construction an \textbf{auxiliary diagram} with respect to $f:(X,B)\to Y$ and $\phi\in\mathbb C(X)$.  

In the canonical bundle formula
\[
    K_X + B + \frac{1}{r}(\phi) = f^*(K_Y + B_Y + M_Y),
\]
we will make the canonical choice of $r = b(F,B_F)$ as
\[
    \set{m\in\mathbb N: m(K_F+B_F)\sim 0} = b(F,B_F)\mathbb N
\]
for a general fiber $F$. In this setup, condition (2) in the definition of the klt-trivial fibration is equivalent to $\lceil B_F\rceil\ge 0$ and $h^0(F, \lceil B_F\rceil) = 1$.

    From the construction, we have $\psi^r = h^*\phi$ and the group $G = \operatorname{Aut}(V/X)$ is generated by a  primitive $r$-th root of unity $\zeta$. The action $G\act \mathscr K_V$ is given by $\psi\mapsto \zeta\cdot\psi$.

We first recall Kawamata's theorem:

\begin{thm}\cite[Theorem 5 and Section 5]{11}\label{kaw}
    Let $f:X\to Y$ be a surjective contraction morphism from a compact K\"ahler manifold $X$ to a projective manifold $Y$ satisfying the following properties:
    \begin{enumerate}
        \item There is a Zariski open dense subset $Y_0\subseteq Y$ such that $D = Y-Y_0$ is a snc divisor;
        \item Put $X_0 = f^{-1}(Y_0)$ and $f_0: X_0\to Y_0$. Then $f_0$ is smooth;
        \item The local monodromies of $R^nf_{0*}\mathbb C_{X_0}$ around $D$ are unipotent, where $n = \operatorname{dim}(X) - \operatorname{dim}(Y)$.
    \end{enumerate}
    Then $f_*\omega_{X/Y}$ is a locally free sheaf and semi-positive.
\end{thm}

In his paper \cite{11}, Kawamata established this result for projective varieties in \cite[Theorem 5]{11}. The proof is based on an analytic estimate which also works for K\"ahler domains. In fact, later in section 5, Kawamata extends this result to the case where $X$ is K\"ahler in \cite[Proof of Corollary 25]{11}. 

The key is that when the fibration is regular enough, $M_Y$ occurs as a direct summand of $f_*\omega_{X/Y}$, and Hodge theory implies that it satisfies the desired properties.

\begin{lem}\label{5.2}
    Consider an auxiliary diagram
    \[
    \begin{tikzcd}
        (X,B)\arrow[swap]{d}{f} & (V,B_V)\arrow{ld}{g}\arrow[swap,dashed]{l}{h}\\
        Y
    \end{tikzcd}
    \]
    Then we have
    \begin{enumerate}
        \item $h:(V,B_V)\to Y$ satisfies condition (1)(3) in the definition of the klt-trivial fibration, and it  induces the same discriminant divisor $B_Y$ and moduli divisor  $M_Y$ as those induced by $f:(X,B)\to Y$;
        \item The group $G$ acts on $h_*\omega_{V/Y}$. The eigensheaf corresponds to the eigenvalue $\zeta$ is 
        \[
            \mathscr L = f_*\mathscr O_X(\lceil -B+f^*B_Y+f^*M_Y\rceil)\cdot\psi.
        \]
        \item If $h:V\to Y$ is semi-stable in codimension one, then $M_Y$ is a $\mathbb Z$-divisor, $\mathscr L$ is semi-positive, and $\mathscr L = \mathscr O_Y(M_Y)\cdot\psi$. 
    \end{enumerate}
\end{lem}

\begin{proof}
    This is \cite[Lemma 5.2]{8}.  We will follow the notations in the proof of \cite[Lemma 5.2]{8}.
    
    (1) The map $h$ is the composition 
    \[
        (V,B_V)\dashrightarrow (\tilde X,B_{\tilde X})\to (X,B)
    \]
    where the first map $d:V\to \tilde X$ is bimeromorphic and the second morphism $\pi:\tilde X\to X$ is finite. By Proposition \ref{finite-base-change}, we have $\tilde f:(\tilde X,B_{\tilde X})\to Y$ and $f$ induce the same moduli and discriminant divisors. Since $(V,B_V)\dashrightarrow (\tilde X,B_{\tilde X})$ is a crepant map, we know $\tilde f$ and $h$ induce the same moduli and discriminant divisors.

    (2) By the projection formula, we have
    \[
        \tilde f_*\omega_{\tilde X/Y} = f_*\pi_*(\omega_{\tilde X}\otimes \pi^*f^*\omega_{Y}^\vee) = f_*(\pi_*\omega_{\tilde X}\otimes f^*\omega_Y^\vee).
    \]
    
    Since the group $G$ acts on $\tilde f_*\omega_{\tilde X/Y}$, by \cite[2.49-2.52]{KM98} and \cite[8.10.3]{6}, we have an eigensheaf decomposition
    \[
    \tilde f_*\omega_{\tilde X/Y} = \bigoplus_{i=0}^{b-1} f_*\mathscr O_X(\lceil (1-i)K_{X/Y}-iB +if^*B_Y+if^*M_Y\rceil)\cdot \psi^i
\]
where $\mathscr O_X(D)\psi^i$ is the subsheaf of  $\mathscr K_{\tilde X}$, generated by $\pi^*a\cdot \psi^i$ where $a\in \mathscr O_X(D)$.

Since $(f,B,\Sigma_X,\Sigma_Y)$  satisfies standard normal crossing assumptions, we know $K_X - f^*K_Y$ is a $\mathbb Z$-divisor and  $B-f^*(B_Y+M_Y)$ has snc support. We conclude that $\tilde X$ has rational singularities, and hence $h_*\omega_{V/Y} = \tilde f_*\omega_{\tilde X/Y}$. Then the eigensheaf decomposition implies that
\[
    \mathscr L = f_*\mathscr O_X(\lceil -B+f^*B_Y+f^*M_Y\rceil)\cdot\psi
\]
is the eigensheaf corresponds to $\zeta$.

(3) Assume $h$ is semi-stable in codimension one. By Example \ref{ssr-discriminant},  there is a big open $Y^\dagger\subseteq Y$ such that $(-B_V+h^*B_Y)|_{h^*(Y^\dagger)}$ is effective and supports no codimension one fiber of $h$. Regard $\psi$ as a meromorphic section of $h_*\mathscr O_V(K_{V/Y})$. Since  $\psi$ is an eigenfunction of $\zeta$, we have $\psi$ is a meromorphic section of $\mathscr L$. Since $\mathscr L$ has rank $1$, we deduce that $\mathscr L\subseteq \mathbb C(Y)h_*(\mathscr O_V \cdot \psi)$.

For any $a\in \mathbb C(Y)$ such that $(a) + M_Y \ge 0$, we have
\[
    (h^*a\cdot \psi)+K_{V/Y} = h^*((a)+M_Y) + (-B_V+h^*B_Y).
\]
Since $(-B_V+h^*B_Y)|_{h^*(Y^\dagger)}\ge 0$, we have $\mathscr O_Y(M_Y)|_{Y^\dagger}\cdot\psi\subseteq h_*\mathscr O_V(K_{V/Y})|_{Y^\dagger}$ and hence $\mathscr O_Y(M_Y)|_{Y^\dagger}\cdot\psi\subseteq\mathscr L|_{Y^\dagger}$. 

Conversely, if $h^*a\cdot \psi$ is a section of $\mathscr L \subseteq \mathbb C(Y)\cdot\psi$, then from (2), we have $h^*a\cdot \psi$ is a section of $h_*\mathscr O_V(K_{V/Y})$. This implies
\[
    0\le (h^*a\cdot \psi) + K_{V/Y} = h^*[(a)+M_Y] + [-B_V+h^*B_Y]. 
\]
Since $(-B_V+h^*B_Y)|_{h^*(Y^\dagger)}$  supports no codimension one fiber of $h$, we have $(a)+M_Y\geq 0$, hence $\mathscr L\subseteq \mathscr O_Y(M_Y)\cdot \psi$. We conclude that $\mathscr L|_{Y^\dagger}\subseteq \mathscr O_Y(M_Y)\cdot \psi|_{Y^\dagger}$, and hence $\mathscr L|_{Y^\dagger} = \mathscr O_Y(M_Y)\cdot \psi|_{Y^\dagger}$ on $Y^\dagger$. 

Since $Y^\dagger\subseteq Y$ is an big open subset, the reflexive hull $\mathscr L^{**} = \mathscr O_Y(M_Y)\cdot \psi$. By Theorem \ref{kaw}, we have $h_*\mathscr O_V(K_{V/Y})$ is locally free and semi-positive. By (2), since $\mathscr L$ is a direct summand of $h_*\mathscr O_V(K_{V/Y})$, we know $\mathscr L = \mathscr L^{**}$ is locally free and $\mathscr L$ is nef.
\end{proof}

Next, we will discuss base changes of auxiliary diagrams. 
Given an auxiliary diagram
    \[
    \begin{tikzcd}
        (X,B)\arrow[swap]{d}{f} & (V,B_V)\arrow{ld}{g}\arrow[swap,dashed]{l}{h}\\
        Y
    \end{tikzcd}
    \]
    and $\gamma:Y'\to Y$ be a generically finite morphism from a smooth projective variety $Y'$ such that there exists a reduced divisor $\Sigma_{Y'}$ supporting $\gamma^{-1}(\Sigma_Y)$ and the locus where $\gamma$ is not \'etale, there is an induced auxiliary diagram $(V',B_{V'})\dashrightarrow (X',B')\to Y'$ such that the following diagram is commutative
\[\begin{tikzcd}
	& {V} && {V'} \\
	{X} && {X'} \\
	& {Y} && {Y'}
	\arrow["{g}"', from=1-2, to=2-1, dashed]
	\arrow["{h}" pos=0.3, from=1-2, to=3-2]
	\arrow[from=1-4, to=1-2]
	\arrow["{g'}"', from=1-4, to=2-3, dashed]
	\arrow["{h'}", from=1-4, to=3-4]
	\arrow["{f}", from=2-1, to=3-2]
	\arrow[from=2-3, to=2-1]
	\arrow["{f'}", from=2-3, to=3-4]
	\arrow["\tau"', from=3-4, to=3-2]
\end{tikzcd}\]

\begin{lem}\label{semi-stable-triangle}
    Given an auxiliary diagram
    \[
    \begin{tikzcd}
        (X,B)\arrow[swap]{d}{f} & (V,B_V)\arrow{ld}{g}\arrow[swap,dashed]{l}{h}\\
        Y
    \end{tikzcd}
    \]
    There exists a finite Galois covering $\tau:Y'\to Y$ from a smooth projective variety $Y'$ such that 
    \begin{enumerate}
        \item there exists a reduced divisor $\Sigma_{Y'}$ supporting $\gamma^{-1}(\Sigma_Y)$ and the locus where $\gamma$ is not \'etale;
        \item The induced morphism $h:V'\to Y'$ is semi-stable in codimension one.
    \end{enumerate}
\end{lem}

\begin{proof}
    The key is that our base $Y$ is projective so that Kawamata's covering trick works. The following is exactly the same as \cite[Proposition 5.4]{8}.
    
    Let $N$ be an integer satisfying the condition in Theorem \ref{SSR}. Since $Y$ is a projective variety, by Kawamata's covering trick (\cite[Theorem 4.1]{8}), there exists a Galois covering $\tau:Y'\to Y$ such that $\tau^*(\Sigma_Y)$ is divided by $N$ and there exists a snc divisor $\Sigma_{Y'}$ containing $\tau^{-1}(\Sigma_Y)$ and the non-\'etale locus. By Theorem \ref{SSR}, the main component of the base change $V\times_YY'$ admits a  resolution that is semistable in codimension one.
\end{proof}

\begin{lem}\label{Ambro-5.5}
    Under the same assumptions as in the previous lemma. Let $\gamma:Y'\to Y$ be a generically finite morphism from a smooth projective variety $Y'$. Assume there exists a snc divisor $\Sigma_{Y'}$ on $Y'$ that contains $\gamma^{-1}(\Sigma_Y)$ and the non-\'etale locus of $\gamma$. Let $M_{Y'}$ be the moduli part of $(X',B')\to Y'$. Then we have $\gamma^*M_Y = M_{Y'}$.
\end{lem}

\begin{proof}
    We will follow Ambro's argument in \cite[Proposition 5.5]{8}.

    \textbf{Step 1:} Assume that both $V\to Y$ and $V'\to Y$ are semistable in codimension one. Then we have $M_Y$ and $M_{Y'}$ are $\mathbb Z$-divisors. Since $h:V'\to V$ is semistable in codimension one, by naturality of the canonical extension, we have
    \[
        h'_*\omega_{V'/Y'} = \gamma^*h_*\omega_{V/Y}.
    \]
    Since by Lemma \ref{5.2}, both $M_Y$ and $M_{Y'}$ are $\zeta$-eigensheaves, we deduce that $\gamma^*\mathscr O_Y(M_Y) =\mathscr O_{Y'} (M_{Y'})$. Since the principal divisors $(\phi)$ and $\gamma^*(\phi)$ are fixed, we deduce that $\gamma^*M_Y-M_{Y'}$ is $\gamma$-exceptional, and hence $\gamma^*M_{Y} = M_{Y'}$.

    \textbf{Step 2:} In general, for a general generically finite base change $\gamma$, we can find a semi-stable model $\gamma'$ of it. Let $\tau:\overline Y\to Y$ be a semi-stable base change from Lemma \ref{semi-stable-triangle}. Since $\overline{Y}\to Y$ is a projective morphism between projective varieties, we can apply \cite[Remark 4.2]{8} to obtain a commutative diagram
    \[
        \begin{tikzcd}
            \overline{Y}\arrow{d}{\tau} & \overline{Y}'\arrow[swap]{l}{\gamma'}\arrow{d}{\tau'}\\
            Y & Y'\arrow{l}{\gamma}
        \end{tikzcd}
    \]
    such that 
    \begin{enumerate}
        \item $\tau'$ is a finite covering and $\gamma'$ is a projective morphism;
        \item $\overline{Y}'$ is smooth projective;
        \item There exists a snc divisor $\Sigma_{Y'}$ such that $\tau'$ is \'etale away from $\Sigma_{Y'}$ and $\tau'^{-1}(\Sigma_Y)_{red}$ is snc, and $\gamma^{-1}(\Sigma_{Y})\subseteq \Sigma_{Y'}$.
    \end{enumerate}
    From the construction, $\overline V\to \overline Y$ is semi-stable in codimension one. By Theorem \ref{SSR} and Kawamata's covering trick, there is a finite covering of $\overline{Y}'$ such that the induced fibration is semi-stable in codimension one. We will replace $\overline{Y}'$ by this finite base change so that $\overline{V}'\to \overline{Y}'$ is semi-stable in codimension one.  We would like to emphasize that this base change construction takes place entirely on the projective base, so that the hyperplane section trick works without difficulty.
    
    By Step 1, we deduce that $M_{\overline{Y}'} = \gamma'^*(M_{\overline{Y}})$. Since $\tau,\tau'$ are finite coverings, we have $\tau^*(M_{Y}) = M_{\overline{Y}}$ and $\tau'^*(M_{Y'}) = M_{\overline Y'}$. This implies that  $\gamma^*M_{Y} = M_{Y'}$.
\end{proof}

\begin{proof}[Proof of Theorem \ref{moduli_b-nef}]
    From the definition of klt trivial fibration, we write
    \[
        K_X+B+\frac{1}{b}(\phi) = f^*(K_Y+B_Y+M_Y)
    \]
    where $\phi\in \mathbb C(X)$ and $b$ is the minimal integer such that the equation holds.

    By Lemma \ref{standard-normal-crossing} and Lemma \ref{auxillary-base-change}, after making a base change $Y'\to Y$, there exists an auxiliary diagram
\[
    \begin{tikzcd}
        (X',B')\arrow[swap]{d}{f'} & (V',B_{V'})\arrow{ld}{g'}\arrow[swap,dashed]{l}{h'}\\
        Y'
    \end{tikzcd}
    \]
    such that $X'$ is still  K\"ahler (See Remark \ref{kaehlerness-perserving}).
    
We claim that $\bd M$ and $\bd K+\bd B$ descends to $Y'$. In fact, for a further (projective) birational base change $\nu: Y''\to Y'$, let $\tilde{Y}''\to (Y'',\nu^{-1}\Sigma_{Y'})$ be a log resolution:
\[
    \begin{tikzcd}
        Y'' \arrow{d}[swap]{\nu} & \tilde{Y}''\arrow{ld}{\nu'}\arrow{l}\\
        Y'
    \end{tikzcd}
\]
Then, by Lemma \ref{Ambro-5.5}, we have $\nu'^*(K_{Y'}+\bd B_{Y'}) = K_{\tilde Y''} + \bd B_{\tilde Y''}$ and $\nu'^*\bd M_{Y'} = \bd M_{\tilde Y''}$. Since all morphisms in the triangle are birational and $\nu'$ factors through $\nu$, we conclude that $\nu^*(K_{Y'}+\bd B_{Y'}) = K_{Y''}+\bd B_{Y''}$.

Next, we will show that $\bd M_{Y'}$ is a nef $\mathbb Q$-Cartier divisor. In fact, by Theorem \ref{SSR}, there exists a finite Galois base change $\tau: \overline{Y}'\to Y'$ which is \'etale away from an snc divisor that contains $\Sigma_Y$. We have an induced auxillary triangle:
\[\begin{tikzcd}
	& {V'} && {\overline{V}'} \\
	{X'} && {\overline{X}'} \\
	& {Y'} && {\overline{Y}'}
	\arrow["{g'}"', from=1-2, to=2-1, dashed]
	\arrow["{h'}" pos=0.3, from=1-2, to=3-2]
	\arrow[from=1-4, to=1-2]
	\arrow["{\overline{g}'}"', from=1-4, to=2-3, dashed]
	\arrow["{\overline h'}", from=1-4, to=3-4]
	\arrow["{f'}", from=2-1, to=3-2]
	\arrow[from=2-3, to=2-1]
	\arrow["{\overline{f}'}", from=2-3, to=3-4]
	\arrow["\tau"', from=3-4, to=3-2]
\end{tikzcd}\]
in which, blowing up $\overline{V}'$ if needed, $\overline{V}'\to \overline{Y}'$ is semi-stable in codimension 1. By Lemma \ref{5.2}, we have $\bd M_{\overline Y'}$ is a nef $\mathbb Z$-divisor. Since $\tau^*\bd M_{Y'} = \bd M_{\overline Y'}$, it follows that $\bd M_{Y'}$ is nef as well.
\end{proof}

As a corollary, we know that assumption (4) in Theorem \ref{thmA} holds when $a(X) =1$:

\begin{prop}\label{curvebase}
    Let $f:(X,B)\to Y$ be a klt-trivial fibration from a K\"ahler pair to a curve $Y$. Then the moduli $\mathbb Q$-divisor $M_Y$ is semiample. 
\end{prop}

\begin{proof}
    This is the K\"ahler version of \cite[Theorem 4.5, Theorem 0.1]{8}. The proof still works in this set-up.
\end{proof}

\section{The Reduction Argument}

 We need the following result on cohomology and base change:

\begin{lem}
    Let $f:X\to Y$ be a proper morphism between analytic varieties and $\mathscr F$ be a coherent sheaf on $X$. Then, there exists a dense open $U\subseteq Y$ such that
    \begin{enumerate}
        \item $R^if_*\mathscr F$ is locally free on $U$ of rank equal to $\operatorname{rank}R^if_*\mathscr F$;
        \item $R^if_*\mathscr F\otimes \mathbb C(y)\to H^i(X_y,\mathscr F|_{X_y})$ is an isomorphism for $y\in U$.
    \end{enumerate}
\end{lem}

\begin{proof}
    By generic flatness, there is a Zariski open $V\subseteq Y$ such that $X_{V}\to V$ is flat. As a result, the function $y\mapsto h^i(X_{y},\mathscr F_{X_{y}})$ is upper semicontinuous on $V$. Since every $\mathbb N$-valued upper semicontinuous function must achieve its minimal in a nonempty open subset, there exists a Zariski open $U$ of $V$ such that $y\mapsto h^i(X_{y},\mathscr F_{X_{y}})$ is a constant function (Cf. \cite[Théorème 3]{MR422670}).
    By Grauert's theorem, we know (1)(2) hold for $y\in U$.
\end{proof}

We first show that, since $\mathbb C(X) = \mathbb C(V)$,  divisors on $X$ should be rigid along the vertical direction.

\begin{lem}\label{Kvertical}
    Let $(X,\dt)$ be a K\"ahler klt pair such that $K_X+\dt$ is nef and $a(X)\neq 0$. Let $X\dashrightarrow V$ be the algebraic reduction map of $X$ and $M$ be the normalization of its graph.  Consider the diagram
    \[
        \begin{tikzcd}
            M\arrow{r}{f}\arrow[swap]{d}{\mu}&V\\
            X
        \end{tikzcd}
    \]
    Assume

    \begin{enumerate}
        \item The algebraic reduction map $X\dashrightarrow V$ is almost holomorphic; and
        \item Log abundance for K\"ahler klt pairs holds in dimension $\operatorname{dim}(X)-a(X)$.
    \end{enumerate}

    Then, for a very general fiber $W$ of $f$, we have $\kappa(W,K_W+\dt_W) = 0$ where $K_W+\dt_W = \mu^*(K_X+\dt)|_W$. In fact, there is a vertical $\mathbb Q$-divisor $E$ (not necessarily effective) on $M$ such that $K_M + \dt_M\sim_\mathbb Q E$.
\end{lem}

\begin{proof}
    Let $W$ be a general fiber and write $\dt_W = \dt|_W$. Then by assumption (1), we have $\dt_W\ge 0$, and hence $(W,\dt_W)$ is a minimal klt log variety. By assumption (2), since $K_W+\dt_W$ is nef, it is semiample, and hence $\kappa(W,K_W+\dt_W) \ge 0$.
 Thus, it suffices to show that $\kappa(W,K_W+\dt_W) \leq 0$.  

    Let $m\geq 1$  be a positive integer such that $m(K_M+\dt_M)$ is an integral divisor and denote $\mathscr F_m = f_*\mathscr O_M(m(K_M + \dt_M))$. We claim that $\operatorname{rank}(\mathscr F_m)\leq 1$ for all $m\geq 1$ such that $m(K_M+\dt_M)$ is an integral divisor.
    
     In fact, by generic flatness and cohomology and base change, for any $m$, there exists a dense open $V_m\subseteq Y$ such that
     \[
        \mathscr F_m\otimes \mathbb C(y) \cong H^0(X_y,m(K_X+\dt)|_{X_y})
     \]
     for every $y\in V_m$, and $\mathscr F_m|_{V_m}$ is a locally free sheaf. If $\mathscr F_m|_{V_m} = 0$ for all $m$, we are done. Thus, we will assume $\mathscr F_m|_{V_m}$ is nontrivial for divisible enough $m$. 
     
     In this case, $\mathscr F_m$ induces the following map over the base $V$:
    \[
        \begin{tikzcd}
            M\arrow[dotted]{r}{g}\arrow[swap]{d}{f}& \mathbb P(\mathscr F_m)\arrow{dl}{\pi}\\
            V
        \end{tikzcd}
    \]
    From the diagram, we have the following inequalities:
    $$a(M)  \geq a(g(M)) = \operatorname{dim}(g(M)) \geq \operatorname{dim}(V) = a(M)$$

    The equality $a(g(M)) = \operatorname{dim}(g(M))$ comes from the fact that both $V$ and $\mathbb P(\mathscr F_m)$ are algebraic. Since the two sides are the same as each other, this forces all the inequalities to be equations. As a result, we have $\operatorname{dim}(g(M)) = \operatorname{dim}(V)$ hence $\operatorname{rank}(\mathscr F_m)\leq 1$.  
    
    Let $V^\circ = \cap_m V_m$. Since each $V_m$ is nonempty and dense in $V$, by Baire's category theorem, their intersection $V^\circ$ is dense as well. For $y\in V^\circ$, we denote $W = f^{-1}(y)$. From the construction, we have $\kappa(W,K_W+\dt_W) \leq 0$.

    Next, we will prove $K_M+\dt_M$ is $\mathbb Q$-linearly equivalent to a vertical divisor.

    We first claim that there exists a $\mathbb Q$-divisor $D$ such that $K_M+\dt_M\sim_\mathbb Q D$ and $D^h\ge 0$. In fact, fix a divisible enough $m$. Then for  $y\in Y$ in a dense subset, write $X_y = W$. Since $\kappa(W,\dt_W)\ge 0$, we have $H^0(W,m(K_W+\dt_W)) \ne 0$ for divisible enough $m$. Since $\mathscr F_m|_{V_m}$ is a locally free sheaf of rank equal to $\operatorname{rank}(\mathscr F_m)$, we have $\operatorname{rank}(\mathscr F_m)\ge 1$. Let $\mathscr L$ be a very ample line bundle on $Y$. Then,  the projection formula implies
    \[
        H^0(M,f^*\mathscr L^{\otimes n}\otimes \mathscr O_X(m(K_M+\dt_M))) = H^0(V,\mathscr L^{\otimes n}\otimes \mathscr F_m) \ne 0
    \]
    for large enough $n$. This shows that 
    \[
        K_M+\dt_M + f^*H\sim_\mathbb Q D,
    \]
    for some $D\ge 0$ on $X$ and some hyperplane section $H$ on $Y$.
    In other words, we conclude that $K_M+\dt_M\sim_\mathbb Q D$ with $D^h\ge 0$.

    Then, it suffices to show $D^h = 0$. In fact, for $y\in \cap_m V_m$, write $X_y = W$. Then, we have  $K_W + \dt_W \sim_\mathbb Q D_h|_W\geq 0$. Since $(W,\dt_W)$ is a minimal K\"ahler klt log variety of dimension $\operatorname{dim}(X) - a(X)$, by assumption (2), we deduce that $|D^h|_W|_\mathbb Q$ is base point free. Since $\kappa(W,D_h|_W) =\kappa(W,K_W+\dt_W) = 0$, the only possibility is that $D_h|_W = 0$. This means $D_h = 0$ because the divisor $D_h$ is horizontal. We conclude that $K_M+\dt_M\sim_\mathbb Q D_v$ where $D_v$ is vertical.
\end{proof}

We will adapt the argument from \cite[Theorem 0.1]{5} to our setup. Based on the previous lemma, we further show that the canonical bundle formula applies after a suitable base change:

\begin{lem}
    Let $(X,\dt)$ be a minimal K\"ahler klt pair satisfying the conditions (1--3) in Theorem \ref{thmA}. Then, there is a diagram
        \[
        \begin{tikzcd}
            M\arrow{r}{f}\arrow[swap]{d}{\mu}&V\\
            X
        \end{tikzcd}
    \]
    where $\mu:M\rightarrow X$ is a projective bimeromorphic contraction from a K\"ahler manifold $M$ to $X$, and $V$ is a projective manifold, such that:

    \begin{enumerate}[(i)]
        \item $f:(M,\dt_M)\rightarrow V$ is a klt-trivial fibration where $K_M+\dt_M=\mu^*(K_X + \dt)$;
        \item $\bd M$ descends to $V$ and there is a contraction $h: V\rightarrow Z$ of projective varieties such that $\bd M_V \sim_\mathbb Q h^*N$ where $N$ is a big and nef $\mathbb Q$-divisor on $Z$; and 
        \item If $E\subseteq M$ is an $f$-exceptional prime divisor, i.e. $\operatorname{codim}f(E)\geq 2$, then $E$  is $\mu$-exceptional.
    \end{enumerate}
\end{lem}

\begin{proof}
    Let $M'$ be the normalization of the graph of the algebraic reduction map $X\dashrightarrow  V'$ of $X$, and $\mu':M'\to X$ and $f':M'\to V'$ be the corresponding morphisms. We write $K_{M'} + \dt_{M'} = \mu'^*(K_X + \dt)$.
        \[
        \begin{tikzcd}
            M'\arrow{r}{f'}\arrow[swap]{d}{\mu'}&V'\\
            X
        \end{tikzcd}
    \]

    By lemma \ref{Kvertical}, there is a vertical and nef $\mathbb Q$-divisor $D$ on $M'$ such that $K_{M'} + \dt_{M'}\sim_\mathbb Q D$. By Corollary \ref{verticalnef}, there is a b-nef b-divisor $\bd D$ on $V'$ such that  $\overline{K_{M'} + \dt_{M'}}\sim_\mathbb Q f^*\bd D$. 

    Then, we will make a sequence of projective birational base changes as follows:
    \[
        \begin{tikzcd}
            M\arrow{r}{\nu}\arrow{d}{f}&M''\arrow{r}\arrow{d}{f''}&M'\arrow{d}{f'}\arrow{r}{\mu'}&X\\
    V\arrow{r}{\pi}&V''\arrow{r}&V'
        \end{tikzcd}
    \]

    Firstly, $V''\rightarrow V'$ is a birational base change such that $\bd D$ descends to a $\mathbb Q$-Cartier $\mathbb Q$-divisor $\bd D_{V''}$ on $V''$,  and $M''$ is the normalization of the main component of the fiber product $M'\times_{V'}V''$. 
     Since $f''$ agrees with $f'$ away from the exceptional locus of $V''\to V'$ and $\dt_{M'}^h\ge0$, we have  $\dt_{M''}|_W\geq 0$ for a general fiber $W$ of $f''$. Since $K_{M''} + \dt_{M''} \sim_\mathbb Q (f'')^*\bd D_{V''}$, we conclude that $f''$ is a klt-trivial fibration, which satisfies the condition (i). 
    
    Next, by assumption (3), the moduli part of $f''$ is b-good and nef. Replacing $V''$ by a higher model if necessary, we can assume there is a contraction $h: V''\rightarrow Z$ of projective varieties such that $\bd M_{V''} \sim_\mathbb Q h^*N$ where $N$ is a big and nef $\mathbb Q$-divisor on $Z$, which satisfies condition (ii).

    Finally, let $\pi:V\rightarrow V''$ be a smooth projective resolution of $V''$  and $M$ be a smooth K\"ahler resolution of the main component of $V\times_{V''} M''$ such that $f:M\to V$ dominates a flat model of $f''$. Then, if $E$ is an $f$-exceptional prime divisor, since $f$ factors through a flat model of $f'$, we know $E$ is an exceptional divisor of the bimeromorphic morphism $M'\to M$. We conclude that $E$ is also an exceptional divisor of the bimeromorphic morphism $M\to X$. This shows that the condition (iii) holds.
\end{proof}

\begin{proof}[Proof of Theorem \ref{thmA}]
    Consider the diagram in the previous lemma
    \[
        \begin{tikzcd}
            M\arrow{r}{f}\arrow[swap]{d}{\mu}&V\\
            X
        \end{tikzcd}
    \]
    satisfying (i)(ii)(iii). The canonical bundle formula reads:
    $$K_M + \dt_M\sim_\mathbb Q f^*(K_V +\bd B_V +\bd M_V).$$

    Write $\dt_M = \dt^+_M - \dt^-_M$. Denote by $\set{\dt_i}$ the prime components of $\dt^-_M$ whose images $D_i := f(\dt_i)$ have codimension one, and by $\set{a_i}$ the coefficients of the $\set{\dt_i}$ in $\set{\dt^-_M}$. Let $F = \sum b_i D_i$ be the smallest $\mathbb Q$-divisor such that  $\sum a_i\dt_i - f^*F$ does not support the fiber  of each $D_i$ in a open neighborhood. Write $A := \dt^-_M-f^*F$. 

From the construction, there exists a big open subset $V^\circ$ of $V$ such that
\begin{enumerate}
    \item $A$ is unsaturated over $V$: for any any prime divisor $E$ on $V$, we have $A$ does not support $f^{-1}(E|_{V^\circ})$;
    \item  $A|_{f^{-1}V^\circ}\geq 0$. 
\end{enumerate}

    Then, for a general $y\in V$, we still have $(\dt_M^+-A)|_{X_y}\geq 0$ hence $f:(M,\dt_M^+-A)\rightarrow V$ is still a klt-trivial fibration such that the moduli part $\bd M$ can be chosen to be unchanged. Let $B'$ be the new discriminant divisor. Then 
    \[
        K_M + \dt_M^+ - A\sim_{\mathbb Q} f^*(K_V + B' + \bd M_V)  
    \]
    Since $A$ is unsaturated over $V$, we have $B'\geq 0$. We conclude that $(V,B')$ is a klt log variety.   
    Since $\bd M_{V} = h^*N$ where $N$ is big and nef, there is a $\mathbb Q$-divisor $\dt_V\geq 0$ on $V$ such that $\dt_V\sim_\mathbb Q B' + \bd M_V$ and $(V,\dt_V)$ is a klt log variety. 

    Write $A = A^+ - A^-$ such that $A^+,A^-\geq 0$ with no common irreducible component. Then we have 
    \[
        \mu^*(K_X + \dt) + \dt_M^- + A^-\sim_{\mathbb Q}f^*(K_V + \dt_V) + A^+.
    \]

    From our construction $A^+$ is unsaturated over $V$.
    We claim that $\dt_M^-+ A^-$ is $\mu$-exceptional. In fact, since $A$ is effective over a big open subset on $V$, we know $A^-$ is $f$-exceptional. By condition (iii), it follows that $A^-$ is $\mu$-exceptional. From the definition of $\dt_M$, it is also clear that $\dt_M^-$ is also $\mu$-exceptional.

    Now, since $K_X + \dt$ is nef, its Zariski decomposition exists and the semi-postive part is $\bd P(K_X+\dt) = \overline{K_X+\dt}$. Since $\dt_M^-+A^-$ is $\mu$-exceptional, it is in particular unsaturated over $X$. By proposition \ref{pullback}, we have $\mu^*(K_X + \dt) + \dt_M^- + A^-$ has a Zariski decomposition and the semi-positive part is $\overline{K_X+ \dt}$.

    As a result, the right hand side $f^*(K_V + \dt_V) + A^+$ has a Zariski decomposition and $\bd P(f^*(K_V + \dt_V) + A^+) = \overline{K_X+ \dt}$. Since $A^+$ is unsaturated, by proposition \ref{pullback}, we know $K_V + \dt_V$ has a Zariski decomposition and 
    $$f^*\bd P(K_V+\dt_V)\sim_\mathbb Q\overline{K_X+ \dt}$$

    By assumption (1),  abundance holds for the projective klt pair $(V,\dt)$ which implies that $\bd P(K_V+\dt_V)$ is b-semiample. Consequently, we have $\overline{K_X+ \dt} \sim f^*\bd P(K_V+\dt_V)$ is b-semiample as well. It follows that $K_X + \dt$ is semiample because $K_X+\dt$ is nef.
\end{proof}

\begin{proof}[Proof of Theorem \ref{dim4}]
    By \cite[Lemma 2.12]{MR4698899},  the algebraic reduction map is almost holomoprphic. By \cite[Proposition 4.16]{loginov2025finitenessprojectivepluricanonicalrepresentation}, we know the moduli part is b-semiample.  Since the abundance conjecture is known for projective threefolds and curves, the assumption (1) is satisfied. By Theorem \ref{thmA}, we conclude that if $K_X$ is nef, then $K_X$ is semiample.
\end{proof}

\begin{proof}[Proof of Theorem \ref{dim3}]
    If $a(X) = 3$, we have $X$ is projective. Assume $a(X) = 1$ or $2$.
    By \cite[Corollary 1.3]{MR2103314}, the algebraic reduction is almost holomorphic. The assumption (3) in Theorem \ref{thmA} is known by \cite[Proposition 4.16]{loginov2025finitenessprojectivepluricanonicalrepresentation} and Proposition \ref{curvebase}. Since the abundance conjecture is known for surfaces and curves, the assumption (1) is satisfied. Theorem \ref{thmA} then implies that if $K_X+\dt$ is nef, then $K_X+\dt$ is semiample.
\end{proof}

\bibliographystyle{plain} 
\bibliography{references}

\end{document}